\documentclass[11pt,oneside,reqno]{article}
\usepackage[margin=1.3in]{geometry}

\usepackage{amsmath,amsthm,amssymb,color}
\usepackage[authoryear]{natbib}
\usepackage{booktabs}

\usepackage{bbm}
\usepackage{mathrsfs}
\usepackage[breaklinks=true]{hyperref}
\usepackage{graphicx}
\usepackage{tikz}
\usetikzlibrary{arrows, automata}
\usetikzlibrary{calc}
\usetikzlibrary{positioning}

\numberwithin{equation}{section}
\allowdisplaybreaks[4]

\theoremstyle{plain}
\newtheorem{theorem}{Theorem}[section]

\newtheorem{lemma}{Lemma}[section]
\newtheorem{corollary}{Corollary}[section]

\theoremstyle{definition}

\newtheorem{remark}{Remark}[section]

\makeatletter

\newcount\minute
\newcount\hour
\newcount\hourMins
\def\now{%
\minute=\time%
\hour=\time \divide \hour by 60%
\hourMins=\hour \multiply\hourMins by 60%
\advance\minute by -\hourMins%
\zeroPadTwo{\the\hour}:\zeroPadTwo{\the\minute}%
}
\def\zeroPadTwo#1{\ifnum #1<10 0\fi#1}

\renewcommand{\cite}{\citet}

\def\^#1{\ifmmode {\mathaccent"705E #1} \else {\accent94 #1} \fi}
\def\~#1{\ifmmode {\mathaccent"707E #1} \else {\accent"7E #1} \fi}

\def\*#1{#1^\ast}
\edef\-#1{\noexpand\ifmmode {\noexpand\bar{#1}} \noexpand\else \-#1\noexpand\fi}
\def\>#1{\vec{#1}}
\def\.#1{\dot{#1}}

\def\atop{\@@atop}
\def\*#1{\mathscr{#1}}

\renewcommand{\leq}{\leqslant}
\renewcommand{\geq}{\geqslant}
\newcommand{\eps}{\varepsilon}
\renewcommand{\eps}{\varepsilon}

\newcommand{\eq}{\eqref}

\newcommand{\diag}{{\mathop{\mathrm{diag}}}}

\newcommand{\IE}{\mathbbm{E}}
\newcommand{\IP}{\mathbbm{P}}
\newcommand{\Var}{\mathop{\mathrm{Var}}\nolimits}
\newcommand{\Cov}{\mathop{\mathrm{Cov}}}

\newcommand{\tr}{\mathop{\mathrm{tr}}}

\def\be#1{\begin{equation*}#1\end{equation*}}
\def\ben#1{\begin{equation}#1\end{equation}}
\def\bes#1{\begin{equation*}\begin{split}#1\end{split}\end{equation*}}
\def\besn#1{\begin{equation}\begin{split}#1\end{split}\end{equation}}

\def\bm#1{\begin{multline*}#1\end{multline*}}
\def\bmn#1{\begin{multline}#1\end{multline}}
\def\ba#1{\begin{align*}#1\end{align*}}
\def\ban#1{\begin{align}#1\end{align}}

\def\norm#1{\Vert#1\Vert}

\def\mid{\vert}

\def\angle#1{{\langle#1\rangle}}

\def\beqn#1\eeqn{\begin{align}#1\end{align}}
\def\beq#1\eeq{\begin{align*}#1\end{align*}}

\usepackage{graphicx}
\usepackage{latexsym}
\usepackage{amsmath,amsthm,amssymb,amscd}
\usepackage{epsf,amsmath}

\def\E{{\IE}}
\def\P{{\IP}}

\newcommand{\tha}{\tilde h_\alpha}

\newcommand{\ol}[1]{\overline{#1}}
\newcommand{\mcl}[1]{\mathcal{#1}}
\newcommand{\bs}[1]{\boldsymbol{#1}}

\DeclareMathOperator{\influence}{Inf}

\renewcommand\section{\@startsection {section}{1}{\z@}%
{-3.5ex \@plus -1ex \@minus -.2ex}%
{1.3ex \@plus.2ex}%
{\center\small\sc\mathversion{bold}\MakeUppercase}}

\def\subsection#1{\@startsection {subsection}{2}{0pt}%
{-3.5ex \@plus -1ex \@minus -.2ex}%
{1ex \@plus.2ex}%
{\bf\mathversion{bold}}{#1}}

\def\subsubsection#1{\@startsection{subsubsection}{3}{0pt}%
{\medskipamount}%
{-10pt}%
{\normalsize\itshape}{\kern-2.2ex. #1.}}

\def\blfootnote{\xdef\@thefnmark{}\@footnotetext}

\makeatother

\begin{document}

\title{New Error Bounds in Multivariate Normal Approximations via Exchangeable Pairs with Applications to Wishart Matrices and Fourth Moment Theorems}
\author{Xiao Fang and Yuta Koike}
\date{\it The Chinese University of Hong Kong and The University of Tokyo} 
\maketitle

\noindent{\bf Abstract:} We extend Stein's celebrated Wasserstein bound for normal approximation via exchangeable pairs to the multi-dimensional setting. 
As an intermediate step, we exploit the symmetry of exchangeable pairs to obtain an error bound for smooth test functions.
We also obtain a continuous version of the multi-dimensional Wasserstein bound in terms of fourth moments.
We apply the main results to multivariate normal approximations to Wishart matrices of size $n$ and degree $d$, where we 
obtain the optimal convergence rate $\sqrt{n^3/d}$ 
under only moment assumptions,
and to degenerate $U$-statistics and Poisson functionals, where we strengthen a few of the fourth moment bounds in the literature on the Wasserstein distance.

\medskip

\noindent{\bf AMS 2010 subject classification: }  60F05, 62E17

\noindent{\bf Keywords and phrases:}  Central limit theorem, fourth moment theorem, exchangeable pairs, Wishart matrix, Stein's method

\section{Introduction and Main Results}\label{sec1}

Let $W$ be a random variable with $\E(W)=0$ and $\Var(W)=1$.
\cite{St86} introduced the following exchangeable pair approach to proving central limit theorems for $W$ with error bounds.
Suppose we can construct another random variable $W'$ on the same probability space 
such that $(W,W')$ and $(W',W)$ have the same distribution (exchangeable), and moreover,
\be{
\E(W'-W|W)=-\lambda W
}
for some positive constant $\lambda$ (linearity condition).
Then we have (cf. \cite[Theorem 1, Lecture III]{St86} and \cite[Theorem 4.9]{ChGoSh11}):
\bes{
d_{\mathcal{W}}(W,Z):=&\sup_{h\in\mathrm{Lip}(\mathbb{R},1)}|\E h(W)-\E h(Z)|\\
\leq & \sqrt{\frac{2}{\pi}}\E\Big|   \E \big[1-\frac{1}{2\lambda} (W'-W)^2 | W \big]  \Big|+\frac{1}{2\lambda}\E|W'-W|^3,
}
where $d_{\mathcal{W}}$ denotes the Wasserstein distance, $Z\sim N(0,1)$, and $\mathrm{Lip}(\mathbb{R},1)$ denotes the set of 1-Lipschitz functions $h$ on $\mathbb{R}$.

Stein's exchangeable pair approach and its variants have found wide applications in normal approximations. These applications include, but are not limited to, the binary expansion of a random integer (\cite{Di77} and \cite[Lecture IV]{St86}); the anti-voter model (\cite{RiRo97}); the representation theory of permutation groups (\cite{Fu05}); character ratios (\cite{ShSu06}); the Erd\"os-Kac theorem (\cite{Ha09}); the Curie-Weiss model (\cite{ChFaSh13}); combinatorial central limit theorems (\cite{ChFa15}); and degenerate $U$-statistics (\cite{DoPe17}). \cite{ChSh11} extended the approach to non-normal approximations and \cite{ShZh19} used the approach to obtain optimal error bounds on the Kolmogorov distance for both normal and non-normal approximations.

\medskip

\textbf{Basic setting.}\quad  Stein's exchangeable pair approach has been extended to the multi-dimensional setting.
Let $d\geq 2$ be an integer.
We follow the general setting of \cite{ReRo09} and assume that for a $d$-dimensional random vector $W$,
we can construct another random vector $W'$ on the same probability space such that 
\be{
\mathcal{L}(W,W')=\mathcal{L}(W',W),
} 
and moreover,
\ben{\label{1}
\E[W'-W|\mathcal{G}]=-\Lambda (W+R)
}
for some invertible $d\times d$ matrix $\Lambda$ and $\sigma$-algebra $\mathcal{G}$ containing $\sigma(W)$.
Gaussian approximation results and error bounds for such $W$ have been obtained by, for example, \cite{ChMe08} and \cite{ReRo09}.
However, the existing error bounds mostly apply to smooth function distances (excluding those results in \cite{ChMe08} which make the special assumption of a continuous underlying symmetry). 
Although we can deduce a Wasserstein bound from these results, such deduced bound is in general non-optimal. Our first main result is a Wasserstein bound assuming the existence of fourth moments.
The optimality of the bound, in terms of the ``sample size", is illustrated by applications to sums of independent random vectors below and to homogeneous sums in Section~\ref{sec2.2}.
Proofs of the main results stated in this section are given in Section~\ref{sec3.1}.

\begin{theorem}[Wasserstein bound] \label{t3}
Let $(W, W')$ be an exchangeable pair of $d$-dimensional random vectors satisfying the approximate linearity condition \eq{1}.
Assume that $\E|W|^4<\infty$.
Let $D=W'-W$. 
Also, let $\Sigma$ be a $d\times d$ positive definite symmetric matrix and define the random matrix $E$ by
\ben{\label{2}
E:=\frac{1}{2} \E[(\Lambda^{-1} D) D^{\top}|\mathcal{G}]-\Sigma.
}
Then we have
\ba{
d_{\mathcal{W}}(W,Z)&:=\sup_{h\in\mathrm{Lip}(\mathbb{R}^d,1)}|\E h(W)-\E h(Z)|\\
&\leq \E|R|+\|\Sigma^{-1/2}\|_{op}\sqrt{\frac{2}{\pi}}\E\|E\|_{H.S.}\\
&\quad+\|\Sigma^{-1/2}\|_{op}^{3/2}\left(\frac{\pi}{8}\right)^{1/4}(\E|W|^2\vee\tr(\Sigma))^{1/4}\sqrt{\E[|\Lambda^{-1}D||D|^3]},
}
where $\mathrm{Lip}(\mathbb{R}^d,1)$ is the set of all 1-Lipschitz functions on $\mathbb{R}^d$, $Z\sim N(0,\Sigma)$ is a $d$-dimensional centered Gaussian vector with covariance matrix $\Sigma$,  $|\cdot|$ denotes the Euclidean norm, $\norm{\cdot}_{op}$ denotes the operator norm, and $\norm{\cdot}_{H.S.}$ denotes the Hilbert-Schmidt norm.
\end{theorem}

\begin{remark}
In Theorem~\ref{t3}, we implicitly assumed that $\E(W)\approx 0$ and $\Cov(W)\approx \Sigma$. Otherwise, $\E|R|$ and $\E\norm{E}_{H.S.}$ are not small and the bound is not useful.
We need $\Sigma$ to be non-singular to carry out integration by parts in the proof (cf. \eq{eq:idp}). 
In the case that $\E(W)=0$, $\Cov(W)=I_d$, where $I_d$ denotes the $d\times d$ identity matrix,
and $Z$ is a standard $d$-dimensional Gaussian vector, the bound reduces to 
\be{
\E|R|+\sqrt{\frac{2}{\pi}} \E \norm{E}_{H.S.}+\left(\frac{\pi}{8}\right)^{1/4} d^{1/4} \sqrt{\E[|\Lambda^{-1}D||D|^3]}.
}
\end{remark}

\medskip

\textbf{Sums of independent random vectors.}\quad We first apply Theorem~\ref{t3} to sums of independent random vectors to illustrate the order of magnitude of the error bound.
Let $W=\frac{1}{\sqrt{n}}\sum_{i=1}^n X_i$ be a $d$-dimensional random vector, where $\{X_1, \dots, X_n\}$ are independent, $\E (X_i)=0$ for each $i\in [n]:=\{1,\dots, n\}$ and $\Cov(W)=I_d$.
A standard construction of exchangeable pairs is as follows. 
Let $\{X_1^*,\dots, X_n^*\}$ be an independent copy of $\{X_1, \dots, X_n\}$.
Let $I\sim \text{Unif}[n]$ be an uniform random index that is independent of $\{X_1, \dots, X_n, X_1^*,\dots, X_n^*\}$. Let $W'=W-\frac{1}{\sqrt{n}}(X_I-X_I^*)$.
It is straightforward to verify that 
\be{
\E (W'-W|W)=-\frac{1}{n} W
}
and $E$ in \eq{2} can be computed as
\be{
E=\frac{1}{2}\big(\frac{1}{n}\sum_{i=1}^n X_i X_i^{\top}-\Sigma \big).
}
Applying Theorem~\ref{t3} with $\Sigma=I_d$ and $Z\sim N(0,I_d)$, we obtain
\be{
d_{\mathcal{W}}(W,Z)\leq constant\cdot \Big\{ \big[\frac{1}{n^2} \sum_{i=1}^n \sum_{j,k=1}^d \Var(X_{ij} X_{ik}) \big]^{1/2}+d^{1/4} \big[ \frac{1}{n^2} \sum_{i=1}^n \E |X_i-X_i^*|^4 \big]^{1/2} \Big\},
}
where $\{X_{ij}: 1\leq j\leq d\}$ are the components of $X_i$. 
For the typical case where $|X_i|\sim \sqrt{d}$ and $|X_{ij}|\sim 1$, the bound reduces to 
\be{
\sim \sqrt{\frac{d^{5/2}}{n}}.
}
This bound has optimal dependence on $n$. The dependence on the dimension $d$ is generally the same as in \cite[Eq.(7)]{Bo20}, who obtained a Wasserstein-2 bound in a comparatively complicated way (see also Theorem 2.5 in \cite{AnGa20}). \cite{Zh18} obtained a Wasserstein-2 bound $\sim \sqrt{d^2/n}\cdot\log n$ when $|X_i|$s are uniformly bounded by $constant \cdot \sqrt{d}$ and showed that his bound is optimal up to the $\log n$ factor; see also Theorem 1 in \cite{ElMiZh18}, where the factor $\log n$ is improved to $\sqrt{\log n}$. It is unclear what is the optimal dependence on $d$ under only a moment condition.


\medskip

As an intermediate step in proving Theorem~\ref{t3}, we exploit the symmetry of exchangeable pairs to obtain an error bound for smooth test functions. 
We introduce a few symbols. 
For an $r$-times differentiable function $f:\mathbb{R}^d\to\mathbb{R}$, we denote by $\nabla^rf(x)$ the $r$-th derivative of $f$ at $x$ regarded as an $r$-linear form on $\mathbb{R}^d$: The value of $\nabla^rf(x)$ evaluated at $u_1,\dots,u_r\in\mathbb{R}^d$ is given by
\[
\langle \nabla^r f(x),u_1\otimes\cdots\otimes u_r\rangle=\sum_{j_1,\dots,j_r=1}^d\partial_{j_1,\dots,j_r}f(x)u_{1,j_1}\cdots u_{r,j_r},
\]
where $\partial_{j_1,\dots,j_r}f(x)$ is a shorthand notation for $\frac{\partial^rf}{\partial x_{j_1}\cdots\partial x_{j_r}}(x)$. When $u_1=\cdots=u_r=:u$, we write $u_1\otimes\cdots\otimes u_r=u^{\otimes r}$ for short. 
We define the \textit{injective norm} of an $r$-linear form $T$ on $\mathbb{R}^d$ by
\[
|T|_\vee:=\sup_{|u_1|\vee\cdots\vee|u_r|\leq1}|\langle T,u_1\otimes\cdots\otimes u_r\rangle|.
\] 
Then, for an $(r-1)$-times differentiable function $h:\mathbb{R}^d\to\mathbb{R}$, we set
\[
M_r(h):=\sup_{x\neq y}\frac{|\nabla^{r-1}h(x)-\nabla^{r-1}h(y)|_\vee}{|x-y|}.
\]
Note that $M_r(h)=\sup_{x\in\mathbb{R}^d}|\nabla^{r}h(x)|_\vee$ if $h$ is $r$-times differentiable. 
See Section 5 of \cite{Ra19} for more details about these symbols. 

\begin{remark}
For a function $h:\mathbb{R}^d\to\mathbb{R}$ of the form $h(x)=\cos(\xi\cdot x)$ or $h(x)=\sin(\xi\cdot x)$ for some $\xi\in\mathbb{R}^d$, where $\cdot$ denotes the Euclidean inner product, we have $M_r(h)=|\xi|^r$ for every $r\geq1$. Therefore, given a sequence $W(n)$ of random vectors in $\mathbb{R}^d$ and another random vector $Y$ in $\mathbb{R}^d$, if we have $|\E h(W(n))-\E h(Y)|\to0$ as $n\to\infty$ for any bounded $C^\infty$ function $h:\mathbb{R}^d\to\mathbb{R}$ with $M_r(h)<\infty$ for every $r\geq1$, then we obtain the point-wise convergence of the characteristic functions of $W(n)$ to that of $Y$: $|\E \exp(\sqrt{-1}\xi\cdot W(n))-\E \exp(\sqrt{-1}\xi\cdot Y)|\to0$ as $n\to\infty$ for all $\xi\in\mathbb{R}^d$. In particular, if $d$ is fixed, then $W(n)$ converge in law to $Y$ as $n\to\infty$. 
\end{remark}

\begin{theorem}[Smooth function bound with improved dimension dependence] \label{t1}
Under the assumptions of Theorem \ref{t3},
we have, for second-order differentiable functions $h$ such that $M_1(h), M_3(h)<\infty$,
\be{
|\E h(W)-\E h(Z)|\leq M_1(h)\left(\E|R|+\norm{\Sigma^{-1/2}}_{op}\sqrt{\frac{2}{\pi}}\E \norm{E}_{H.S.}\right)+\frac{\norm{\Sigma^{-1/2}}_{op}}{12\sqrt{2\pi}} M_3(h) \E[|\Lambda^{-1}D||D|^3].
}
\end{theorem}

\begin{remark}\label{r1}
By a simple modification of the proof of Theorem~\ref{t1}, we can obtain an alternative bound
\bes{
|\E h(W)-\E h(Z)|\leq & \sup_{0\leq t\leq 1} \Big\{ \E \big[|\nabla h(\sqrt{t}W+\sqrt{1-t}Z)|^2 \big] \Big\}^{1/2} \Big(\sqrt{\E|R|^2}+\norm{\Sigma^{-1/2}}_{op}\sqrt{\E \norm{E}_{H.S.}^2} \Big)\\
&+\frac{\norm{\Sigma^{-1/2}}_{op}}{12\sqrt{2\pi}} M_3(h) \E[|\Lambda^{-1}D||D|^3],
}
which is useful in the case where $|\nabla h(x)|$ is not uniformly bounded.
\end{remark}

\begin{remark}
Comparing with smooth function bounds in the literature, e.g., \cite[Theorem 2.3]{ChMe08} and \cite[Theorem 2.1]{ReRo09}, the bound in Theorem~\ref{t1} has improved dependence on dimension. This can be easily checked by examining the case of sums of independent random vectors above. In fact, it is crucial to use the bound in Theorem~\ref{t1} to obtain the optimal convergence rate for the application to Wishart matrices in Section~\ref{sec2.1}.
\end{remark}

\begin{remark}[Singular covariance matrix]
When $\Sigma$ is not invertible, we get the following alternative bound for any thrice differentiable function $h:\mathbb{R}^d\to\mathbb{R}$ with $\E|h(W)|+\E|h(Z)|<\infty$:
\ben{\label{singular}
|\E h(W)-\E h(Z)|\leq M_1(h)\E|R|+\frac{\sqrt d}{2}M_2(h)\E \norm{E}_{H.S.}+\frac{M_4(h)}{64} \E[|\Lambda^{-1}D||D|^3].
}
This should be compared to the following bound from \cite[Theorem 3.1]{DoPe17} (see also \cite[Theorem 3]{Me09}): For any twice differentiable function $h:\mathbb{R}^d\to\mathbb{R}$ with $\E|h(W)|+\E|h(Z)|<\infty$,
\ben{\label{singular2}
|\E h(W)-\E h(Z)|\leq M_1(h)\E|R|+\frac{\sqrt d}{2}M_2(h)\E \norm{E}_{H.S.}+\frac{M_3(h)}{18} \E[|\Lambda^{-1}D||D|^2].
}
While our bound requires $h$ to be more smooth, the quantity $\E[|\Lambda^{-1}D||D|^3]$ is typically of a smaller order than $\E[|\Lambda^{-1}D||D|^2]$. 
To see the effect of this difference, let us consider the typical case of sums of independent random vectors as above. Then, \eqref{singular} and \eqref{singular2} respectively reduce to 
\[
|\E h(W)-\E h(Z)|\leq C_1\left(M_2(h)\sqrt{\frac{d^{3}}{n}}+M_4(h)\frac{d^2}{n}\right)
\]
and
\[
|\E h(W)-\E h(Z)|\leq C_2\{M_2(h)+M_3(h)\}\sqrt{\frac{d^{3}}{n}},
\]
where $C_1,C_2$ are constants depending only on $\max_{i,j}\E|X_{ij}|^4$. 
Following the proof of Theorem \ref{t3} but ignoring the integration by parts step therein, 
we can derive the corresponding Wasserstein bounds from these results via a smoothing argument. Then, the former yields a bound $\sim(d^4/n)^{1/4}$, while the latter implies a bound $\sim(d^4/n)^{1/4}\vee(d^5/n)^{1/6}$, so our bound always provides a better rate. 
\end{remark}

Our next result is a continuous version of the Wasserstein bound in multivariate normal approximations. The setting was introduced by \cite{DoViZh18} and proved to be useful in the study of Gaussian, Poisson and Rademacher functionals. In the special case that $\rho_j(W)=0$ in \eq{03}, the result reduces to those in \cite{ChMe08} and \cite{NoZh17}.
An application of Theorem~\ref{t4} to Poisson functionals is given in Section~\ref{sec2.3}.

\begin{theorem}[Continuous version of the Wasserstein bound] \label{t4}
For every $t>0$, let $(W, W_t)$ be an exchangeable pair of $d$-dimensional random vectors such that $\E|W|^4<\infty$ and
\ben{\label{01}
\lim_{t\downarrow0}\frac{1}{t}\E[W_t-W|\mathcal{G}]=-\Lambda (W+R)\quad\mathrm{in}~L^1(\P)
}
for some invertible $d\times d$ (non-random) matrix $\Lambda$, $d$-dimensional random vector $R$, and $\sigma$-algebra $\mathcal{G}$ containing $\sigma(W)$. 
Suppose also that there is a $d\times d$ positive definite symmetric matrix $\Sigma$ and a $d\times d$ random matrix $S$ satisfying
\ben{\label{02}
\lim_{t\downarrow0}\frac{1}{t} \E[(W_t-W)(W_t-W)^{\top}|\mathcal{G}]=2\Lambda \Sigma+S\quad\mathrm{in}~L^1(\Omega,\|\cdot\|_{H.S.}).
}
Moreover, suppose that, for every $j\in\{1,\dots,d\}$, there is a constant $\rho_j(W)$ satisfying
\ben{\label{03}
\limsup_{t\downarrow0}\frac{1}{t}\E((W_t)_j-W_j)^4\leq\rho_j(W).
}
Then we have
\begin{multline*}
d_{\mathcal{W}}(W,Z)
\leq \E|R|+\frac{\norm{\Sigma^{-1/2}}_{op}}{\sqrt{2\pi}}\E\|\Lambda^{-1}S\|_{H.S.}\\
+\norm{\Sigma^{-1/2}}_{op}^{3/2}\left(\frac{\pi}{8}\right)^{1/4}(\E|W|^2\vee\tr(\Sigma))^{1/4}\sqrt{d}\sqrt{\|\Lambda^{-1}\|_{op}\sum_{j=1}^d\rho_i(W)},
\end{multline*}
where $Z\sim N(0,\Sigma)$.
\end{theorem}


\section{Applications}\label{sec2}

In this section, we present three applications of our main results. Their proofs are deferred to Section~\ref{sec3.2}.
We begin with multivariate normal approximations to Wishart matrices.

\subsection{Wishart matrices}\label{sec2.1}

Let $X=\{X_{ik}: 1\leq i\leq n, 1\leq k\leq d\}$ be a matrix with i.i.d.\ entries such that $\E X_{11}=0, \E X_{11}^2=1$ and $\E X_{11}^4<\infty$. For $1\leq i<j\leq n$, let 
\be{
W_{ij}=\frac{1}{\sqrt{d}} \sum_{k=1}^d X_{ik} X_{jk}
}
be the upper diagonal entries of the Wishart matrix $\frac{1}{\sqrt{d}} X X^{\top}$. We are interested in approximating $W=\{W_{ij}: 1\leq i<j \leq n\}$, regarded as an ${n \choose 2}$-vector, by a standard Gaussian vector $Z$ when both $n$ and $d$ grow to infinity.

In the case where $X_{11}$ follows the standard Gaussian distribution, \cite{JiLi15} and \cite{Bu16} proved that the total variation distance between $W$ and $Z$ tends to zero if $d\gg n^3$ and tends to one if $d\ll n^3$ (see also \cite{RaRi19}).
\cite{BuGa18} generalized the result to the case where $X_{11}$ follows a log-concave distribution.
\cite{NoZh18} considered row-wise i.i.d.\ Gaussian matrices $X$ where each row is a Gaussian vector with a general covariance matrix and \cite{Mi20} considered column-wise i.i.d.\ matrices $X$ where each column follows a log-concave measure on $\mathbb{R}^n$. \cite{NoZh18} and \cite{Mi20} proved convergence of $W$ to $Z$ in the Wasserstein-1 and Wasserstein-2 distances respectively in the asymptotic region $d\gg n^3$. They also considered Gaussian approximations for Wishart tensors.

In \cite{BuGa18} and \cite{Mi20}, it was pointed out that a standard application of Stein's method, e.g. by \cite{ChMe08}, only provides an error bound in the Gaussian approximation for $W$ for smooth test functions that vanishes when $d\gg n^6$ (in fact, $d\gg n^4$ using the exchangeable pair in the proof of Theorem~\ref{t2} below). We use Theorem~\ref{t1} to obtain the optimal convergence rate $\sqrt{n^3/d}$ for the i.i.d.\ case for smooth test functions. Except for the existence of the fourth moment of $X_{11}$, we do not impose any other distributional assumptions. We also note that the proof works for the non-identically distributed case (cf. \eq{12} and \eq{13}) and for Wishart tensors (cf. Remark \ref{Wtensor}).
In the Appendix, we use a modified version of Theorem \ref{t3} to obtain the optimal convergence rate $\sqrt{n^3/d}$ for the Wasserstein distance assuming in addition $X_{11}$ has finite sixth moment.

\begin{theorem}\label{t2}
Let $X=\{X_{ik}: 1\leq i\leq n, 1\leq k\leq d\}$ be a matrix with i.i.d.\ entries such that $\E X_{11}=0, \E X_{11}^2=1$, and $\E X_{11}^4<\infty$.
Regard $W=\{W_{ij}: 1\leq i<j \leq n\}$ as an ${n \choose 2}$-vector where
\be{
W_{ij}=\frac{1}{\sqrt{d}} \sum_{k=1}^d X_{ik} X_{jk}.
}
Let $Z$ be a standard ${n \choose 2}$-dimensional Gaussian vector.
Then we have, for second-order differentiable functions $h$ such that $M_1(h), M_3(h)<\infty$,
\besn{\label{9}
|\E h(W)-\E h(Z)|\leq &M_1(h)\sqrt{\frac{n^2}{2\pi d}[\E X_{11}^4+(\E X_{11}^4)^2] +\frac{n^3}{8\pi d} [3+\E X_{11}^4]}\\
&+M_3(h)\frac{n^3}{12d\sqrt{2\pi}} (EX_{11}^4+3)\left(\frac{EX_{11}^4}{n}+1\right).
}
\end{theorem}

\begin{remark}\label{Wtensor}
Let $p\geq 2$ be an integer.
Under the same assumptions as in Theorem~\ref{t2}, let
\be{
W_{i_1 \dots i_p}=\frac{1}{\sqrt{d}}\sum_{k=1}^d X_{i_1 k}\dots X_{i_p k}
}
be entries of the Wishart tensor. 
Regard $W=\{W_{i_1 \dots i_p}: 1\leq i_1<\dots < i_p \leq n\}$ as an ${n \choose p}$-vector.
Let $Z$ be a standard ${n \choose p}$-dimensional Gaussian vector.
Following the proof of Theorem~\ref{t2}, we can easily obtain (details omitted)
\be{
|\E h(W)-\E h(Z)|\leq C_p \Big\{ M_1(h) \sqrt{\frac{n^{2p-1}}{d}}+  M_3(h) \frac{n^{2p-1}}{d}  \Big\},
}
where $C_p$ is a constant only depending on $p$ and $\E X_{11}^4$.
This recovers the range $d\gg n^{2p-1}$ for asymptotic normality in \cite{NoZh18} and \cite{Mi20} under only moment assumptions.
\end{remark}

\medskip

\textbf{Optimality of the convergence rate.}\quad
By using the alternative bound in Remark~\ref{r1}, we can replace $\sqrt{\frac{2}{\pi}}M_1(h)$ in \eq{9} by 
$\sup_{0\leq t\leq 1} \Big\{ \E \big[|\nabla h(\sqrt{t}W+\sqrt{1-t}Z)|^2 \big] \Big\}^{1/2}$. The resulting bound can be shown to be optimal by considering $h(x)=n^{-3/2}\sum_{1\leq i<j< k\leq n}x_{ij}x_{jk}x_{ik}$. In this case, we have $\E \big[|\nabla h(\sqrt{t}W+\sqrt{1-t}Z)|^2 \big]\leq C\sqrt{\E X_{11}^4+n/d}$ and $M_3(h)\leq C$ for some universal constant $C$, while $\E h(W)=\frac{1}{\sqrt{n^3d}}\binom{n}{3}\sim\sqrt{n^3/d}$ and $\E h(Z)=0$. 

\begin{remark}\label{l2remark}
We can naturally regard $\binom{n}{2}$-random vectors $W$ and $Z$ as random elements of the infinite-dimensional Hilbert space $\ell^2(\mathbb{N})$. 
Then, we may discuss the closeness between their respective laws $P_{n,d}$ and $Q_n$ induced on the space $\ell^2(\mathbb{N})$. However, we need some care because the sequence $\{Q_n\}_{n=1}^\infty$ is \textit{not} tight in this case. 
In such a situation, \cite{DaRo09} argue that it is reasonable to define the concept of weak convergence in the following way: We say $P_{n,d}-Q_n\to0$ weakly as $n,d\to\infty$ if $\int fd(P_{n,d}-Q_n)\to0$ for all bounded and \textit{uniformly} continuous functions $f:\ell^2(\mathbb{N})\to\mathbb{R}$. 
Using Theorem 3 in \cite{DaRo09}, we can easily verify that $P_{n,d}-Q_n\to0$ weakly as $n,d\to\infty$ when $d_{\mcl W}(W,Z)\to0$. 
Applying Theorem \ref{t3} instead of Theorem \ref{t1} in the proof of Theorem \ref{t2}, we obtain $d_{\mcl W}(W,Z)=O(\sqrt{n^4/d})$; hence, $P_{n,d}-Q_n\to0$ weakly as $n,d\to\infty$ if $n^4/d\to0$. 
In the Appendix, we prove $d_{\mcl W}(W,Z)=O(\sqrt{n^3/d}\vee(n^3/d)^{2/3})$ as long as $\E X_{11}^6<\infty$, so we recover the optimal condition $n^3/d\to0$ for the weak convergence in the above sense under the finiteness of the sixth moment. 
\end{remark}

\subsection{Degenerate $U$-statistics}\label{sec2.2}

For a positive integer $n$, we write $[n]:=\{1,\dots,n\}$. For each $i\in[n]$, let $X_i$ be a random variable taking values in a measurable space $(E_i,\mcl{E}_i)$. 
Assume $X_1,\dots,X_n$ are independent. 
We set $\mcl{F}_J:=\sigma(X_j,j\in J)$ for every $J\subset[n]$. Also, we denote by $|J|$ the number of elements in $J$. 

Let $U\in L^1(\P)$ be $\mcl{F}_{[n]}$-measurable. It is known that $U$ admits the \textit{Hoeffding decomposition} of the form
\ben{\label{hoeff}
U=\sum_{J\subset[n]}U^J,
}
where $U^J\in L^1(\P)$ is $\mcl{F}_J$-measurable and satisfies 
\ben{\label{orthogonal}
\E[U^J\mid\mcl{F}_K]=0,
}
whenever $J\not\subset K\subset[n]$. Such a decomposition is almost surely unique and given by
\[
U^J=\sum_{L\subset J}(-1)^{|J|-|L|}\E[U\mid\mcl{F}_L],\qquad J\subset[n].
\]
See Section 1.2 of \cite{DoPe17} and references therein for more details. 
Let $p\in[n]$. We say that $U$ is a \textit{degenerate $U$-statistic of order $p$} if $U^J=0$ whenever $J\subset[n]$ is such that $|J|\neq p$. 
In the celebrated work \cite{deJo90}, de Jong showed that a sequence $U(n)\in L^4(\P)$ of degenerate $U$-statistics of fixed order $p$ with mean 0 and variance 1 converges in law to the standard normal distribution if
\[
\E U(n)^4\to 3\qquad\text{and}\qquad
\varrho(U(n)):=\max_{1\leq i\leq n}\sum_{J\subset[n]:i\in J}\Var[U(n)^J]\to0.
\]
Recently, \cite{DoPe17} have established a quantitative version of the above result in the Wasserstein distance via exchangeable pairs. They have also obtained a multi-dimensional extension of de Jong's CLT together with error bounds for smooth test functions. 
Here, we complement their results by deriving a multi-dimensional Wasserstein bound of de Jong type using Theorem \ref{t3}. 

Let $W$ be an $\mcl{F}_{[n]}$-measurable random vector in $\mathbb{R}^d$. We assume that, for each $j\in[n]$, $W_j$ is a degenerate $U$-statistic of order $p_j$. Without loss of generality, we may assume $p_1\leq\cdots\leq p_d$. 
For every $j\in[d]$, we assume $W_j\in L^4(\P)$ with mean 0 and variance 1. In addition, we set
\[
\varrho_{n,j}^2:=\varrho(W_j)=\max_{1\leq i\leq n}\sum_{J\subset[n]:i\in J}\Var(W_j^J).
\] 

\begin{theorem}\label{thm:dejong}
Under the setting described above, assume $\Sigma:=\Cov(W)$ is invertible. Let $Z\sim N(0,\Sigma)$. 
Then we have
\bmn{\label{eq:dejong}
d_{\mathcal{W}}(W,Z)
\leq \|\Sigma^{-1/2}\|_{op}\sqrt{\frac{2p_d^2}{\pi p_1^2}\left(\E|W|^4-\E|Z|^4\right)+C_{\bs{p}}d\sum_{j=1}^d\varrho_{n,j}^2}\\
+\|\Sigma^{-1/2}\|_{op}^{3/2}\left(\frac{\pi}{8}\right)^{1/4}d^{3/4}\left(\frac{p_d}{p_1}\right)^{3/8}\sqrt{\sum_{j=1}^d\left(8(\E[W_j^4]-3)+K_{j}\varrho_{n,j}^2\right)},
}
where $C_{\bs{p}}>0$ is a constant depending only on $p_1,\dots,p_d$ and $K_j>0$ is a constant depending only on $p_j$.
\end{theorem}

When $d=1$, Theorem \ref{thm:dejong} recovers Theorem 1.3 in \cite{DoPe17} with a possibly different constant. 
Also, note that
\[
|\E|W|^4-\E|Z|^4|\leq\sum_{j=1}^d|\E W_j^4-3|+\sum_{j\neq k}|\E W_j^2W_k^2-\E Z_j^2Z_k^2|.
\]
Therefore, Theorem \ref{thm:dejong} may be seen as a quantitative version of Theorem 1.7 in \cite{DoPe17} in terms of Wasserstein bound in the case of invertible $\Sigma$. 

\medskip

\textbf{Homogeneous sums.}\quad 
A prominent example of degenerate $U$-statistics is a \textit{multilinear homogeneous sum}. 
Suppose that $X_1,\dots,X_n$ are real-valued random variables with mean 0 and variance 1. 
For every $j\in[d]$, we assume $W_j$ is of the form
\[
W_j=\sum_{i_1,\dots,i_{p_j}=1}^nf_j(i_1,\dots,i_{p_j})X_{i_1}\cdots X_{i_{p_j}},
\]
where $f_j:[n]^{p_j}\to\mathbb{R}$ is a symmetric function vanishing on diagonals (i.e.~$f_j(i_1,\dots,i_{p_j})=0$ unless $i_1,\dots,i_{p_j}$ are mutually different). 
Without loss of generality, we may assume $p_1\leq\cdots\leq p_d$. 
Assume also that $W_j$ is normalized so that
\[
\Var(W_j)=p_j!\sum_{i_1,\dots,i_{p_j}=1}^nf_j(i_1,\dots,i_{p_j})^2=1.
\]
For each $i\in[n]$, we define the \textit{$i$-th influence function} of $f_j$ by
\ba{
\influence_i(f_j)&:=\sum_{i_2,\dots,i_{p_j}=1}^nf_j(i,i_2,\dots,i_{p_j})^2.
}
We set $\mcl{M}(f_j):=\max_{1\leq i\leq n}\influence_i(f_j)$. 

\begin{corollary}\label{coro:homo}
Under the setting described above, assume $\Sigma:=\Cov(W)$ is invertible. 
Let $Z\sim N(0,\Sigma)$. 
Then we have
\bmn{\label{eq:homo}
d_{\mathcal{W}}(W,Z)
\leq \|\Sigma^{-1/2}\|_{op}\sqrt{\frac{2p_d^2}{\pi p_1^2}\sum_{j,k=1\atop j\leq k}^d\Delta_{j,k}+C'_{\bs{p}}d\sum_{j=1}^dM^{p_j}\mcl{M}(f_j)}\\
+\|\Sigma^{-1/2}\|_{op}^{3/2}\left(\frac{\pi}{8}\right)^{1/4}d^{3/4}\left(\frac{p_d}{p_1}\right)^{3/8}\sqrt{\sum_{j=1}^d\left(8(\E[W_j^4]-3)+K'_{j}\mcl{M}(f_j)\right)},
}
where $M:=\max_{1\leq i\leq n}(\E X_i^4)$, $C'_{\bs{p}}>0$ is a constant depending only on $p_1,\dots,p_d$, $K'_j>0$ is a constant depending only on $p_j$, and
\bm{
\Delta_{j,k}:=1_{\{p_j<p_k\}}C_jM^{p_j/2}\sqrt{(\E W_k^4-3)+C_kM^{p_k}\mcl{M}(f_k)}\\
+1_{\{p_j=p_k\}}\left(2(\E W_j^4-3)+C_jM^{p_j}\mcl{M}(f_j)\right)
}
with $C_j>0$ a constant depending only on $p_j$ for each $j\in[d]$. 
\end{corollary}

\begin{remark}
It is worth mentioning that, unlike the bound \eqref{eq:dejong} for general degenerate $U$-statistics, the bound \eqref{eq:homo} does not contain any joint moments like $\E W_j^2W_k^2$ with $j\neq k$. 
In particular, $W$ converges in law to $N(0,\Sigma)$ if
\ben{\label{dejong}
\max_{1\leq j\leq d}|\E W_j^4-3|\to0\qquad
\text{and}\qquad
\max_{1\leq j\leq d}\mcl{M}(f_j)\to0,
}
provided $\Cov(W)\to\Sigma$ and $M=O(1)$. 
This fact is not new and follows from Lemma 4.3 and Theorem 7.2 in \cite{NoPeRe10} for example. 
It is interesting to observe that \eqref{dejong} seemingly concerns only the coordinate-wise information of $W$. 
This type of phenomenon was first discovered in \cite{PeTu05} when $X_i$ are Gaussian, where even the second convergence in \eqref{dejong} is dropped. 
In fact, it is known that the second convergence is implied by the first one in \eqref{dejong} for several classes of distributions of $X_i$. 
\cite{NPPS16} use this fact to prove the joint asymptotic normality of $W$ is implied by the marginal asymptotic normality; see Theorem 4.1 ibidem.    
\end{remark}

\medskip

\textbf{Optimality of the bound.}\quad
The dependence of the bound \eqref{eq:homo} on the quantities in \eqref{dejong} is generally optimal. To see this, let us assume that $d=1$, $p_1=2$ and $X_i$ are standard Gaussian. Also, assume that $n/3$ is integer and the matrix $A:=(f_1(i,j))_{1\leq i,j\leq d}$ is given by
\ba{
A=\diag(\underbrace{B,\dots,B}_{n/3}),\qquad
\text{where }
B:=\frac{1}{2\sqrt{n}}\begin{pmatrix}
0 & 1 & 1 \\
1 & 0 & 1 \\
1 & 1 & 0
\end{pmatrix}.
}
In this case, we have 
\[
\E W^2=\frac{2n}{3}\tr(B^2)=1,\quad
\E W^3=\frac{8n}{3}\tr(B^3)=\frac{2}{\sqrt n}
\]
and
\[
\E W^4-3=\frac{48n}{3}\tr(B^4)=\frac{12}{n}.
\]
Moreover, since $W$ belongs to the second Wiener chaos of an isonormal Gaussian process over $\mathbb{R}^n$, we infer from the proof of \cite[Theorem 1.2]{NoPe15} 
\[
|\E\sin(W)-\E\sin(Z)|\geq\frac{1}{2\sqrt e}\frac{2}{\sqrt n}-C_1\left(\frac{12}{n}\right)^{1/4}\max\left\{\frac{2}{\sqrt n},\frac{12}{n}\right\},
\]
where $C_1>0$ is a universal constant. Therefore, there is a constant $c>0$ such that $d_{\mathcal{W}}(W,Z)\geq c/\sqrt n$ for sufficiently large $n$. Since $\mathcal{M}(f_1)=1/(2n)$, the bound \eqref{eq:homo} is sharp in terms of $\E W^4-3$ and $\mathcal{M}(f_1)$.

\begin{remark}
When $p_j<p_k$ for some $j,k\in[d]$, the bound \eqref{eq:homo} depends on the fourth roots of the quantities in \eqref{dejong} rather than their square roots. This is a typical phenomenon in the literature of fourth moment theorems; see Remark 1.9(a) in \cite{DoViZh18} for instance. 
\end{remark}

\begin{remark}
\cite{NoPeRe10} have obtained error bounds for multivariate normal approximation of $W$ in their Theorem 7.2 in terms of smooth function distances $|\E h(W)-\E h(Z)|$ with bounded
$\|h''\|_\infty:=\max_{1\leq i,j\leq d}\sup_{x\in\mathbb{R}^d}|\partial_{ij}h(x)|$ and $\|h'''\|_\infty:=\max_{1\leq i,j,k\leq d}\sup_{x\in\mathbb{R}^d}|\partial_{ijk}h(x)|$. 
In view of Lemma 2.1 and Remark 2.3 in \cite{Ko19}, their bound has the same dependence on the quantities in \eq{dejong} as in \eq{eq:homo}. 
However, while it is possible to obtain a Wasserstein bound from their bound by a smoothing argument, this generally leads to suboptimal dependence on the quantities in \eq{dejong} (see e.g.~the proof of \cite[Proposition 5.4]{NoPeRe10}). 
Thanks to Theorem~\ref{t3}, we are able to strengthen the bound in the Wasserstein distance.
\end{remark}

\subsection{Poisson functionals}\label{sec2.3}

In this subsection we apply Theorem \ref{t4} to derive a Wasserstein bound for the \textit{fourth moment theorem} on the Poisson space in the multi-dimensional setting, which strengthens an earlier result obtained by \cite{DoViZh18} in the Wasserstein distance. We refer to Section 1.3 of \cite{DoViZh18} and references therein for unexplained concepts appearing below.
\begin{theorem}\label{t5}
Let $(\mathcal{Z},\mathscr{Z},\mu)$ be a $\sigma$-finite measure space\footnote{We may presumably allow $\mu$ to be an $s$-finite measure in this result; see Remark 1.2.1 and footnote 12 in \cite{Zheng18}. We keep the $\sigma$-finiteness assumption because the results of \cite{DoViZh18} are stated in such a setting.} and let $\eta$ be a Poisson random measure on $(\mathcal{Z},\mathscr{Z})$ with control $\mu$. 
Also, let $1\leq q_1\leq\cdots\leq q_d$ be integers and $W$ be a $d$-dimensional random vector such that $\E W_j^4<\infty$ and $W_j$ belongs to the $q_j$-th Poisson Wiener chaos associated with $\eta$ for all $j=1,\dots,d$. 
Assume $\Sigma:=\Cov(W)$ is invertible. 
Then we have
\begin{multline}\label{pois1}
d_{\mathcal{W}}(W,Z)
\leq \norm{\Sigma^{-1/2}}_{op}\frac{q_d}{q_1}\sqrt{\E[|W|^4-|Z|^4]}\\
+\norm{\Sigma^{-1/2}}_{op}^{3/2}\sqrt{\frac{8q_d}{q_1}}\tr(\Sigma)^{1/4}\sqrt{d}\sqrt{\sum_{j=1}^d\left(\E W_j^4-3(\E W_j^2)^2\right)},
\end{multline}
where $Z\sim N(0,\Sigma)$. Moreover, if $q_1=\cdots=q_d$, we have
\ben{\label{pois2}
d_{\mathcal{W}}(W,Z)
\leq 2\sqrt2\left(\norm{\Sigma^{-1/2}}_{op}+\norm{\Sigma^{-1/2}}_{op}^{3/2}\tr(\Sigma)^{1/4}\right)\sqrt{d}\sqrt{\sum_{j=1}^d\left(\E W_j^4-3(\E W_j^2)^2\right)}.
}
\end{theorem}

\begin{remark}
In the same setting as in Theorem \ref{t5}, \cite{DoViZh18} have essentially obtained the following bound: For any $C^2$ function $g:\mathbb{R}^d\to\mathbb{R}$,
\begin{multline*}
|\E[g(W)]-\E[g(Z)]|\leq
\frac{(2q_d-1)M_1(g)\norm{\Sigma^{-1/2}}_{op}}{\sqrt{2\pi}q_1}\sqrt{\E[|W|^4-\E|Z|^4]}\\
+\frac{\sqrt{2\pi}q_dM_2(g)\norm{\Sigma^{-1/2}}_{op}}{6q_1}\sqrt{d\tr(\Sigma)}\sqrt{\sum_{j=1}^d\left(\E W_j^4-3(\E W_j^2)^2\right)}.
\end{multline*}
We note that there should be an additional factor of $\sqrt{d}$ in their Eq.(3.4).\footnote{This can be checked by examining the proof of their Proposition 3.5, in particular, the first display on page 25.}
Compared to this estimate, the second term of our bound \eqref{pois1} improves the dimension dependence from $d$ to $d^{3/4}$ when $\Sigma=I_d$. 
In addition, our bound does not require the test function $g$ to satisfy $M_2(g)<\infty$. 
\end{remark}

\begin{remark}
Using the exchangeable pairs coupling constructed in \cite{Zh19}, it will also be possible to derive a multi-dimensional ``fourth-moment-influence'' type Wasserstein bound in the Rademacher setting via Theorem \ref{t4}. We omit the details. 
\end{remark}

As a simple illustration, we consider normal approximation of multivariate compound Poisson distributions. Let $X_1,X_2,\dots$ be i.i.d.~isotropic random vectors in $\mathbb{R}^d$ with finite fourth moments and $N=\{N_t\}_{t\geq0}$ be a Poisson process with intensity $\lambda>0$ and independent of $\{X_i\}_{i=1}^\infty$. We take $W:=\lambda^{-1/2}\sum_{i=1}^{N_1}X_i$, which may be seen as an analog of (scaled) sums of i.i.d.~random vectors. Since the coordinates of $W$ belong to the first Poisson Wiener chaos associated with the jump measure of $N$, we can apply Theorem \ref{t5} and obtain
\[
d_{\mathcal{W}}(W,Z)\leq\frac{3\sqrt{2}d^{3/4}}{\sqrt \lambda}\sqrt{\sum_{j=1}^d\E X_{1j}^4}
\leq3\sqrt{2}\sqrt{\max_{1\leq j\leq d}\E X_{1j}^4}\sqrt{\frac{d^{5/2}}{\lambda}},
\]
where $X_{1j}$ denotes the $j$-th component of $X_1$. We observe that the bound depends on the dimension $d$ and ``sample size'' $\lambda$ in an analogous way to the case of sums of i.i.d.~random vectors (cf.~Section \ref{sec1}).

\section{Proofs}

\subsection{Proofs of main results}\label{sec3.1}

In this subsection, we prove our main results stated in Section~\ref{sec1}.
To prove Theorems~\ref{t3} and \ref{t1}, we need the following lemma, which contains our key idea of exploiting the symmetry of exchangeable pairs (cf. \eq{eq:xi}).

\begin{lemma}\label{l1}
Under the assumptions of Theorem \ref{t3}, we have
\ben{\label{15}
|\E[\mathscr{S}f(W)]|
\leq M_1(f)\E|R|+\sup_w\|\mathrm{Hess}f(w)\|_{H.S.}\E\|E\|_{H.S.}
+\frac{M_4(f)}{16}\E[|\Lambda^{-1}D||D|^3]
}
for any third-order differentiable function $f:\mathbb{R}^p\to\mathbb{R}$ with $M_1(f),\sup_w\|\mathrm{Hess}f(w)\|_{H.S.},M_4(f)<\infty$, where
\[
\mathscr{S}f(w):=\langle \Sigma,\mathrm{Hess}f(w)\rangle_{H.S.}-w\cdot\nabla f(w),\qquad w\in\mathbb{R}^d.
\] 
\end{lemma}

\begin{proof}[Proof of Lemma~\ref{l1}]
By exchangeability, Taylor's expansion and assumptions \eq{1} and \eq{2}, we have
\begin{align}
0&=\frac{1}{2}\E[\Lambda^{-1}D\cdot(\nabla f (W')+\nabla f(W))]\nonumber\\
&=\E\left[\frac{1}{2}\Lambda^{-1}D\cdot(\nabla f(W')-\nabla f(W))+\Lambda^{-1}D\cdot\nabla f(W)\right]\nonumber\\
&=\E\left[\frac{1}{2}\sum_{j,k=1}^d(\Lambda^{-1}D)_jD_k\partial_{jk}f(W)+\Xi+\Lambda^{-1}D\cdot\nabla f(W)\right]\label{start}\\
&=\E\left[\langle \Sigma,\mathrm{Hess}f(w)\rangle_{H.S.}+\angle{E, \text{Hess} f(W)}_{H.S.}+\Xi-(W+R)\cdot\nabla f(W)\right],\nonumber
\end{align}
where
\[
\Xi=\frac{1}{2}\sum_{j,k,l=1}^d(\Lambda^{-1}D)_jD_kD_lU\partial_{jkl}f(W+(1-U)D)
\]
and $U$ is a uniform random variable on $[0,1]$ independent of everything else. 
Hence
\begin{align}
|\E[\mathscr{S}f(W)\rangle]|
\leq M_1(f)\E|R|+\sup_w\|\mathrm{Hess}f(w)\|_{H.S.}\E\|E\|_{H.S.}\label{7}
+|\E[\Xi]|.
\end{align}
To estimate $|\E[\Xi]|$, we rewrite it as follows. By exchangeability we have
\begin{align*}
&\E[(\Lambda^{-1}D)_jD_kD_lU\partial_{jkl}f(W+(1-U)D)]\\
&=-\E[(\Lambda^{-1}D)_jD_kD_lU\partial_{jkl}f(W'-(1-U)D)]\\
&=-\E[(\Lambda^{-1}D)_jD_kD_lU\partial_{jkl}f(W+UD)].
\end{align*}
Hence we obtain
\begin{align}
\E[\Xi]&=\frac{1}{4}\sum_{j,k,l=1}^d\E[(\Lambda^{-1}D)_jD_kD_lU\{\partial_{jkl}f(W+(1-U)D)-\partial_{jkl}f(W+UD)\}].\label{eq:xi}
\end{align}
Thus we conclude
\[
|\E[\Xi]|\leq \frac{M_4(f)}{4}\E[|\Lambda^{-1}D||D|^3]\E[U|1-2U|]=\frac{M_4(f)}{16}\E[|\Lambda^{-1}D||D|^3].
\]
Combining this estimate with \eqref{7}, we obtain the desired result. 
\end{proof}

Next, we prove our first main result, Theorem~\ref{t3}. 
Because the test function is not smooth enough, it is a common strategy in Stein's method to smooth the test function first, then quantify the error introduced by such smoothing, finally balance the smoothing error with the smooth test function bound (cf. \eq{15}) to obtain the final result.
There are many smoothing lemmas available in the Stein's method literature, we choose the one by \cite{Ra19} (cf. \eq{14}) to use some readily available results.

\begin{proof}[Proof of Theorem \ref{t3}]
Take a Lipschitz function $h:\mathbb{R}^d\to\mathbb{R}$ arbitrarily.  
For every $\alpha\in(0,\pi/2)$, we define the function $\tilde h_\alpha:\mathbb{R}^d\to\mathbb{R}$ by
\begin{equation*}
\tilde h_\alpha(w)=\int_{\mathbb{R}^d}h(w\cos\alpha+\Sigma^{1/2}z\sin\alpha)\phi_d(z)dz,\qquad w\in\mathbb{R}^d,
\end{equation*}
where $\phi_d$ is the $d$-dimensional standard normal density. It is easy to check that $\tilde h_\alpha$ is infinitely differentiable and
\begin{multline}\label{eq:idp}
\partial_{j_1,\dots,j_r}\tilde h_\alpha(w)=(-1)^{r}\frac{\cos^{r}\alpha}{\sin^{r}\alpha}\int_{\mathbb{R}^d}h(w\cos\alpha+\Sigma^{1/2}z\sin\alpha)\\
\times\sum_{i_1,\dots,i_r=1}^d(\Sigma^{-1/2})_{i_1,j_1}\cdots(\Sigma^{-1/2})_{i_r,j_r}\partial_{i_1,\dots,i_r}\phi_d(z)dz.
\end{multline}
See also the proof of \cite[Lemma 4.6]{Ra19} for an analogous discussion. 
Therefore, noting the inequality after Eq.(4.9) of \cite{Ra19} as well as \cite[Proposition 5.8]{Ra19}, we obtain
\begin{equation}\label{deriv1}
M_{r+1}(\tilde h_\alpha)\leq c_{r}\frac{\cos^{r+1}\alpha}{\sin^{r}\alpha}M_1(h)\|\Sigma^{-1/2}\|_{op}^r
\end{equation}
for any nonnegative integer $r$,
where $c_r:=\int_{-\infty}^\infty|\phi_1^{(r)}(z)|dz$.
In particular, we have by Eq.(4.10) of \cite{Ra19}
\ben{\label{eq:const}
c_0=1\qquad \text{and} \qquad c_3=\frac{2+8e^{-3/2}}{\sqrt{2\pi}}<\frac{4}{\sqrt{2\pi}}.
}
Meanwhile, an analogous argument to the proof of \cite[Lemma 2, Point 4]{Me09} yields
\begin{equation}\label{deriv2}
\sup_w\|\mathrm{Hess}\tilde h_\alpha(w)\|_{H.S.}\leq \sqrt{\frac{2}{\pi}}\frac{\cos^2\alpha}{\sin\alpha}M_1(h) \|\Sigma^{-1/2}\|_{op}.
\end{equation}
Combining \eqref{deriv1}--\eqref{deriv2} with Lemma \ref{l1}, we obtain 
\begin{equation*}
|\E\mathscr{S}\tilde h_\alpha(W)|\leq M_1(h)\left\{\E|R|\cos\alpha
+\|\Sigma^{-1/2}\|_{op}\sqrt{\frac{2}{\pi}}\frac{\cos^2\alpha}{\sin\alpha}\E\|E\|_{H.S.}
+\frac{\|\Sigma^{-1/2}\|_{op}^3}{4\sqrt{2\pi}}\frac{\cos^4\alpha}{\sin^3\alpha}\E[|\Lambda^{-1}D||D|^3]\right\}.
\end{equation*}
Now, we obtain $\frac{d}{d\alpha}\tilde h_\alpha(w)=\mathscr{S}\tilde h_\alpha(w)\tan\alpha$ by differentiation under the integral sign and Gaussian integration by parts. We also have 
\ba{
|\E h(W)-\E \tilde h_\eps(W)|
&\leq M_1(h)\sqrt{(1-\cos\eps)^2\E |W|^2+\E |Z|^2\sin^2\eps}\leq 2AM_1(h)\sin\frac{\eps}{2}
}
for any $\eps\in[0,\pi/2]$, where $A:=\sqrt{\E|W|^2\vee\tr(\Sigma)}$. Consequently, we obtain (cf.~\cite[Eq.(4.14) and (4.23)]{Ra19})
\ben{\label{14}
|\E h(W)-\E h(Z)|\leq\int_\eps^{\pi/2}|\E[\mathscr{S}\tilde h_\alpha(W)]|\tan\alpha ~d\alpha
+2AM_1(h)\sin\frac{\eps}{2}.
}
Therefore, if $\E[|\Lambda^{-1}D||D|^3]=0$, by letting $\eps=0$ we obtain
\[
|\E h(W)-\E h(Z)|\leq M_1(h)\left(\E|R|+\|\Sigma^{-1/2}\|_{op} \sqrt{\frac{2}{\pi}} \E\|E\|_{H.S.}\right).
\]
So we complete the proof. 
Meanwhile, if $\E[|\Lambda^{-1}D||D|^3]>0$, assuming $\eps>0$, we obtain
\ba{
&\int_\eps^{\pi/2}|\E[\mathscr{S}\tilde h_\alpha(W)]|\tan\alpha ~d\alpha\\
&\leq M_1(h)\int_\eps^{\pi/2}\left(\E|R|\sin\alpha
+\|\Sigma^{-1/2}\|_{op}\sqrt{\frac{2}{\pi}}\E\|E\|_{H.S.}\cos\alpha
+\frac{\|\Sigma^{-1/2}\|_{op}^3}{4\sqrt{2\pi}}\frac{\cos^3\alpha}{\sin^2\alpha}\E[|\Lambda^{-1}D||D|^3]\right)d\alpha\\
&\leq M_1(h)\left(\E|R|+\|\Sigma^{-1/2}\|_{op}\sqrt{\frac{2}{\pi}}\E\|E\|_{H.S.}+\frac{\|\Sigma^{-1/2}\|_{op}^3}{4\sqrt{2\pi}}\frac{1}{\sin\eps}\E[|\Lambda^{-1}D||D|^3]\right).
}
Set $\eps:=\frac{1}{2}(\frac{\pi}{8})^{1/4}\sqrt{\|\Sigma^{-1/2}\|_{op}^3\E[|\Lambda^{-1}D||D|^3]/A}$. If $\eps>\pi/2$, we have $2\eps A>2A$. Since we always have $d_{\mathcal{W}}(W,Z)\leq\E|W-Z|\leq2 A$, the desired bound is trivial in this case. Otherwise, we may apply the above estimate with this $\eps$ and obtain
\ba{
|\E h(W)-\E h(Z)&|\leq M_1(h)\left(\E|R|+\|\Sigma^{-1/2}\|_{op}\sqrt{\frac{2}{\pi}}\E\|E\|_{H.S.}+2\eps A\right),
}
where we used Jordan's inequality $2/\pi<\sin\eps/\eps\leq1$. This is the desired bound.
\end{proof}

Next, we prove Theorem~\ref{t1}, which is a consequence of Lemma~\ref{l1} and properties of the solution to the Stein equation \eq{3}.

\begin{proof}[Proof of Theorem~\ref{t1}]
Let 
\ben{\label{6}
f(x):=f_h(x)=\int_0^1 \frac{1}{2t} [\E h(\sqrt{t}x+\sqrt{1-t}Z)-\E h(Z)]dt
}
be the solution to the Stein equation (cf.~Lemma 1, Point 3 of \cite{Me09})
\ben{\label{3}
-\mathscr{S}f(x)=h(x)-\E h(Z).
}
It is easy to check that (cf.~Lemma 2, Point 1 of \cite{Me09})
\ben{\label{4}
M_1(f)\leq M_1(h).
}
It is also known that (cf.~Lemma 2, Point 4 of \cite{Me09})
\ben{\label{5}
\sup_x \norm{\text{Hess} f(x)}_{H.S.}\leq \sqrt{\frac{2}{\pi}}M_1(h)\norm{\Sigma^{-1/2}}_{op}.
}
Moreover, we have by Proposition 2.1 of \cite{Ga16}
\ben{\label{8}
M_4(f)\leq\frac{4}{3\sqrt{2\pi}}M_3(h)\norm{\Sigma^{-1/2}}_{op}.
}
Theorem~\ref{t1} follows from \eq{4}--\eq{8} and Lemma \ref{l1}.
\end{proof}

\begin{proof}[Proof of \eqref{singular}]
Define the function $f$ by \eqref{6} again. We have by Lemma 2, Point 1 of \cite{Me09}
\be{
M_1(f)\leq M_1(h)\qquad\text{and}\qquad M_4(f)\leq \frac{1}{4}M_4(h).
}
Also, we have by Lemma 2, Point 2 of \cite{Me09}
\be{
\sup_x \norm{\text{Hess} f(x)}_{H.S.}\leq
\frac{1}{2}\sup_x \norm{\text{Hess} h(x)}_{H.S.}\leq
\frac{\sqrt d}{2}M_2(h),
}
where the last inequality follows because $M_2(h)=\sup_x \norm{\text{Hess} h(x)}_{op}$. 
\eq{singular} follows from the above bounds and Lemma \ref{l1}. 
\end{proof}

Finally, we prove Theorem~\ref{t4} following the same strategy of the proof of Theorem~\ref{t3}.

\begin{proof}[Proof of Theorem~\ref{t4}]
We need to establish a counterpart of Lemma \ref{l1} in the present setting, which can be shown in line with the proof of Proposition 3.5 in \cite{DoViZh18} as follows. 
Let $f:\mathbb{R}^d\to\mathbb{R}$ be a $C^4$ function with bounded partial derivatives. 
Then, similarly to the derivation of \eqref{start}, we obtain for every $t>0$
\[
0=\E\left[\frac{1}{2}\angle{\Lambda^{-1}D_tD_t^\top, \text{Hess} f(W)}_{H.S.}+\Xi_t+ (\Lambda^{-1}D_t )\cdot\nabla f(W) \right],
\]
where $D_t=W_t-W$ and
\[
\Xi_t=\frac{1}{2}\sum_{j,k,l=1}^d(\Lambda^{-1}D_t)_j(D_t)_k(D_t)_lU\partial_{jkl}f(W+(1-U)D_t)
\]
with $U$ being a uniform random variable on $[0,1]$ independent of everything else. By the same argument as in the derivation of \eqref{eq:xi}, we deduce
\[
\E[\Xi_t]=\frac{1}{4}\sum_{j,k,l=1}^d\E[(\Lambda^{-1}D_t)_j(D_t)_k(D_t)_lU\{\partial_{jkl}f(W+(1-U)D_t)-\partial_{jkl}f(W+UD_t)\}].
\]
Thus we obtain
\ba{
|\E[\Xi_t]|&\leq \frac{M_4(f)}{16}\E[|\Lambda^{-1}D_t||D_t|^3]
\leq \frac{M_4(f)}{16}\|\Lambda^{-1}\|_{op}\E[|D_t|^4]\\
&\leq\frac{M_4(f)}{16}\|\Lambda^{-1}\|_{op}\cdot d\sum_{j=1}^d\E((W_t)_j-W_j)^4.
}
Hence, \eqref{03} yields
\[
\limsup_{t\downarrow0}\frac{1}{t}|\E[\Xi_t]|\leq\frac{M_4(f)}{16}d\|\Lambda^{-1}\|_{op}\sum_{j=1}^d\rho_i(W).
\]
Meanwhile, we obtain from \eqref{01}--\eqref{02}
\ba{
-\lim_{t\downarrow0}\frac{1}{t}\E[\Xi_t]&=\E\left[\frac{1}{2}\angle{2\Sigma+\Lambda^{-1}S, \text{Hess} f(W)}_{H.S.}-(W+R)\cdot\nabla f(W)\right]\\
&=\E\mathscr{S}f(W)+\frac{1}{2}\E\left[\angle{\Lambda^{-1}S, \text{Hess} f(W)}_{H.S.}-R\cdot\nabla f(W)\right].
}
Consequently, we conclude
\begin{multline}\label{est-l1}
|\E\mathscr{S}f(W)|
\leq M_1(f)\E|R|+\frac{1}{2}\sup_w\|\mathrm{Hess}f(w)\|_{H.S.}\E\|\Lambda^{-1}S\|_{H.S.}\\
+\frac{M_4(f)}{16}d\|\Lambda^{-1}\|_{op}\sum_{j=1}^d\rho_i(W).
\end{multline}

Now, the remainder of the proof is completely parallel to that of Theorem \ref{t3} with using \eqref{est-l1} instead of Lemma \ref{l1} (note that the function $\tilde h_\alpha$ in the proof of Theorem \ref{t3} has bounded partial derivatives as long as $\alpha>0$, thanks to \eqref{deriv1}).
\end{proof}

\subsection{Proof of applications}\label{sec3.2}

In this subsection, we prove the results stated in Section~\ref{sec2}.
Theorem~\ref{t2} follows from Theorem~\ref{t1} and a new construction of exchangeable pairs for Wishart matrices.

\begin{proof}[Proof of Theorem~\ref{t2}]
We first construct an exchangeable pair satisfying the linearity condition in \eq{1}. 
Let $X^*=\{X_{ik}^*: 1\leq i\leq n, 1\leq k\leq d\}$ be an independent copy of $X$. 
Let $I\sim \text{Unif}[n]$ and $K\sim \text{Unif}[d]$ be independent uniform random indices that are independent of $X$ and $X^*$. 
Let $X'=\{X_{ik}': 1\leq i\leq n, 1\leq k\leq d\}$ where
\be{
X_{ik}'=
\begin{cases}
X_{ik}^*, & \text{if}\ i=I, k=K\\
X_{ik}, & \text{otherwise}.
\end{cases}
}
Let 
\be{
W_{ij}'=\frac{1}{\sqrt{d}} \sum_{k=1}^d X_{ik}' X_{jk}'
}
and regard $W'=\{W_{ij}': 1\leq i<j \leq n\}$ as an ${n \choose 2}$-vector.
By construction, $\mathcal{L}(X, X')=\mathcal{L}(X', X)$; Hence, $\mathcal{L}(W, W')=\mathcal{L}(W', W)$.
For $1\leq i<j \leq n$, we have
\bes{
&\E(W_{ij}'-W_{ij}|X)\\
=&\E\big[ (W_{ij}'-W_{ij})1(i=I) +(W_{ij}'-W_{ij})1(j=I) |X \big]\\
=&\E\big[\frac{1}{\sqrt{d}} (X_{iK}^*-X_{iK})X_{jK}1(i=I) +\frac{1}{\sqrt{d}} X_{iK}(X_{jK}^*-X_{jK})1(j=I)   | X    \big]\\
=&\frac{1}{nd}\sum_{k=1}^d \E \big[  \frac{1}{\sqrt{d}} (X_{ik}^*-X_{ik})X_{jk} +\frac{1}{\sqrt{d}} X_{ik}(X_{jk}^*-X_{jk})  | X \big]\\
=&-\frac{2}{nd} W_{ij}.
}
Therefore,
\be{
\E (W'-W|X)=-\frac{2}{nd} W
}
and \eq{1} is satisfied with 
\be{
\Lambda=\frac{2}{nd}I_{{n \choose 2}},\quad R=0, \quad \mathcal{G}=\sigma(X).
}
Now we compute $E$ in \eq{2} with $\Sigma=I_{\binom{n}{2}}$.
For $i<j$,
\bes{
E_{ij, ij}=&\frac{nd}{4} \E [(W_{ij}'-W_{ij})^2|X]-1\\
=& \frac{nd}{4} \E \big[ (W_{ij}'-W_{ij})^2 1(i=I) +(W_{ij}'-W_{ij})^2 1(j=I) |X    \big]-1\\
=& \frac{d}{4} \E \big[\frac{1}{d} (X_{iK}^*-X_{iK})^2 X_{jK}^2 +\frac{1}{d} X_{iK}^2 (X_{jK}^*-X_{jK})^2   |X  \big] -1\\
=& \frac{1}{4d} \sum_{k=1}^d (X_{ik}^2+X_{jk}^2+2X_{ik}^2 X_{jk}^2)-1,
}
which has mean zero.
For $i<j<l$ (similarly for other cases of one common index),
\bes{
E_{ij, il}=&\frac{nd}{4} \E [(W_{ij}'-W_{ij})(W_{il}'-W_{il})|X]\\
=&\frac{nd}{4} \E [(W_{ij}'-W_{ij})(W_{il}'-W_{il})1(i=I)|X]\\
=&\frac{d}{4} \E \big[\frac{1}{\sqrt{d}} (X_{iK}^*-X_{iK})X_{jK}\frac{1}{\sqrt{d}} (X_{iK}^*-X_{iK})X_{lK}|X     \big]\\
=&\frac{1}{4d}\sum_{k=1}^d (1+X_{ik}^2) X_{jk} X_{lk},
}
and for $i<j$, $l<m$ such that $\{i, j\}\cap \{l,m\}=\emptyset$,
\be{
E_{ij, lm}=\frac{nd}{4} E[(W_{ij}'-W_{ij})(W_{lm}'-W_{lm})|X]=0.
}
Therefore,
\besn{\label{12}
\E \norm{E}_{H.S.}\leq &\sqrt{\frac{1}{4d^2}\sum_{1\leq i<j\leq n} \sum_{k=1}^d (\E X_{ik}^4 +\E X_{jk}^4+2\E X_{ik}^4 X_{jk}^4) +\frac{n^2}{16d^2} \sum_{i=1}^n \sum_{k=1}^d \E(1+X_{ik}^2)^2}\\
\leq & \sqrt{\frac{n^2}{4d}[\E X_{11}^4+(\E X_{11}^4)^2] +\frac{n^3}{16d} [3+\E X_{11}^4]}.
}
We also have
\be{
D:=W'-W=\frac{1}{\sqrt{d}}(X_{IK}^*-X_{IK})(X_{1K}, \dots, X_{(I-1)K}, X_{(I+1)K}, \dots, X_{nK}, 0,\dots, 0)^\top,
}
where we have transformed $W'-W$ into a vector and put all the zeroes to the end.
Therefore,
\besn{\label{13}
&\E[|\Lambda^{-1}D||D|^3] =\frac{nd}{2} \E |D|^4\\
=&\frac{1}{2d^2} \sum_{i=1}^n \sum_{k=1}^d \E (X_{ik}^*-X_{ik})^4 \E (X_{1k}^2+\dots+X_{(i-1)k}^2+X_{(i+1)k}^2+\dots+X_{nk}^2)^2\\
\leq & \frac{1}{2d^2} \sum_{i=1}^n \sum_{k=1}^d (2\E X_{ik}^4+6) \left(\sum_{j=1}^n \E X_{jk}^4+n^2\right)\\
=& \frac{n^3}{d} (\E X_{11}^4+3) \left(\frac{\E X_{11}^4}{n}+1\right).
}
Theorem~\ref{t1}, together with \eq{12} and \eq{13}, yields \eq{9}.
\end{proof}

Next, we apply the Wasserstein bound in Theorem~\ref{t3} to obtain Theorem~\ref{thm:dejong}.

\begin{proof}[Proof of Theorem~\ref{thm:dejong}]
We apply Theorem \ref{t3} with the help of the results in \cite{DoPe17}. 
Following \cite{DoPe17}, we construct an exchangeable pair $(W,W')$ satisfying condition \eq{1} as follows. 
Let $X^*=\{X^*_1,\dots,X^*_n\}$ be an independent copy of $X:=\{X_1,\dots,X_n\}$. 
Also, let $I\sim \text{Unif}[n]$ be an index independent of $X$ and $X^*$. 
Define $X'=\{X'_1,\dots,X'_n\}$ by
\be{
X_{i}'=
\begin{cases}
X_{i}^*, & \text{if}\ i=I,\\
X_{i}, & \text{otherwise}.
\end{cases}
}
Now, since $W$ is $\mcl{F}_{[n]}$-measurable, there is a function $f:\prod_{i=1}^nE_i\to\mathbb{R}^d$ measurable with respect to the product $\sigma$-field of $\mcl{E}_1,\dots,\mcl{E}_n$ such that $W=f(X_1,\dots,X_n)$. Then we define $W':=f(X'_1,\dots,X'_n)$. 
It is easy to check $\mathcal{L}(X, X')=\mathcal{L}(X', X)$; hence, $\mathcal{L}(W, W')=\mathcal{L}(W', W)$. 
Moreover, we have by Lemma 3.2 of \cite{DoPe17}
\[
\E[W'-W\mid X]=-\Lambda W,
\]
where $\Lambda:=n^{-1}\diag(p_1,\dots,p_d)$. 
Hence, Theorem \ref{t3} yields
\ben{\label{dejong-est}
d_{\mathcal{W}}(W,Z)
\leq \|\Sigma^{-1/2}\|_{op}\sqrt{\frac{2}{\pi}}\E\|E\|_{H.S.}
+\|\Sigma^{-1/2}\|_{op}^{3/2}\left(\frac{\pi}{8}\right)^{1/4}d^{1/4}\sqrt{\E[|\Lambda^{-1}D||D|^3]},
}
where $D:=W'-W$ and $E$ is defined by \eqref{2} with $\mcl{G}=\sigma(X)$. 

Now we estimate quantities on the right-hand side of \eqref{dejong-est}. 
First, we have by Lemma 2.12 of \cite{DoPe17}
\[
\frac{n}{p_j}\E[D_j^4]\leq8(\E[W_j^4]-3)+K_{j}\varrho_{n,j}^2
\]
for every $j\in[d]$, where $K_{j}>0$ is a constant depending only on $p_j$. 
Thus, we obtain by the Schwarz inequality
\besn{\label{dejong-second}
\E[|\Lambda^{-1}D||D|^3]
&\leq n\left(\sum_{j=1}^d\frac{1}{p_j^3}\right)^{1/4}\left(\sum_{j=1}^dp_j\right)^{3/4}\E\left[\sum_{j=1}^d\frac{1}{p_j}D_j^4\right]\\
&\leq d\left(\frac{p_d}{p_1}\right)^{3/4}\sum_{j=1}^d\left(8(\E[W_j^4]-3)+K_{j}\varrho_{n,j}^2\right),
}
where the second inequality follows from $p_1\leq \dots \leq p_d$.
Next, define the random matrix $S=(S_{jk})_{1\leq j,k\leq d}$ by
\[
S:=\E[(W'-W)(W'-W)^{\top}\mid X]-2\Lambda\Sigma.
\]
For any $j,k\in[d]$ and $j\leq k$, we have from Eqs.(3.14)--(3.15) and Propositions 3.5--3.6 of \cite{DoPe17} 
\ba{
n^2\E S_{jk}^2&\leq (p_j+p_k)^2\left(\Cov(W_j^2,W_k^2)
+\min\{p_j\varrho_{n,k}^2,p_k\varrho_{n,j}^2\}+C_{j,k}\max\{\varrho_{n,j}^2,\varrho_{n,k}^2\}
\right)1_{\{p_j<p_k\}}\\
&\quad+4p_j^2(\Cov(W_j^2,W_k^2)-2\Sigma_{jk}^2+p_j\min\{\varrho_{n,k}^2,\varrho_{n,j}^2\}
+p_j\varrho_{n,k}\varrho_{n,j}
+C_{j,k}\max\{\varrho_{n,j}^2,\varrho_{n,k}^2\})1_{\{p_j=p_k\}},
}
where $C_{j,k}>0$ is a constant depending only on $p_j,p_k$. Since $\Sigma_{jk}=0$ if $p_j<p_k$ (cf.~\eqref{orthogonal}), we deduce
\ba{
n^2\E S_{jk}^2
\leq (p_j+p_k)^2\left(\Cov(W_j^2,W_k^2)-2\Sigma_{jk}^2
+C_{j,k}'(\varrho_{n,j}^2+\varrho_{n,k}^2)
\right),
}
where $C_{j,k}'>0$ depends only on $p_j,p_k$. Hence we infer that
\ba{
\E\|E\|_{H.S.}^2
=\frac{1}{4}\sum_{j,k=1}^{d}\frac{n^2\E S_{jk}^2}{p_j^2}
\leq\frac{p_d^2}{p_1^2}\sum_{j,k=1}^{d}\left(\Cov(W_j^2,W_k^2)-2\Sigma_{jk}^2
+C_{j,k}'(\varrho_{n,j}^2+\varrho_{n,k}^2)
\right).
}
Using Eq.(4.2) of \cite{NoRo14}, we obtain
\ba{
\sum_{j,k=1}^{d}\left(\Cov(W_j^2,W_k^2)-2\Sigma_{jk}^2\right)
=\sum_{j,k=1}^{d}\left(\E W_j^2W_k^2-\Sigma_{jj}\Sigma_{kk}-2\Sigma_{jk}^2\right)
=\E|W|^4-\E|Z|^4.
}
Therefore, we conclude that
\ben{\label{dejong-first}
\E\|E\|_{H.S.}\leq\sqrt{\E\|E\|_{H.S.}^2}
\leq\sqrt{\frac{p_d^2}{p_1^2}\left(\E|W|^4-\E|Z|^4\right)+C'_{\bs{p}}d\sum_{j=1}^d\varrho_{n,j}^2},
}
where $C'_{\bs{p}}>0$ depends only on $p_1,\dots,p_d$. 
Plugging \eqref{dejong-second}--\eqref{dejong-first} into \eqref{dejong-est}, we obtain the desired result. 
\end{proof}

\begin{proof}[Proof of Corollary \ref{coro:homo}]
Throughout the proof, for every $j\in[d]$, $C_j$ denotes a positive constant depending only on $p_j$. Note that the value of $C_j$ may change from line to line. 

For any $j\in[d]$ and $J\subset[n]$, we have by the uniqueness of the Hoeffding decomposition \eqref{hoeff}
\[
W_j^J=\left\{
\begin{array}{cl}
p_j!f_j(i_1,\dots,i_{p_j})X_{i_1}\cdots X_{i_{p_j}}&\text{if }|J|=p_j\text{ and }J=\{i_1,\dots,i_p\},\\
0 & \text{if }|J|\neq p_j.
\end{array}\right.
\]
Therefore, $W_j$ is a degenerate $U$-statistics of order $p_j$ and $\varrho_{n,j}^2=p_j!^2\mcl{M}(f_j)$. 
Also, note that $M\geq\max_j\sqrt{\E W_j^2}=1$. 
Therefore, in view of Theorem \ref{thm:dejong}, it remains to prove
\ben{\label{homo:aim}
|\E|W|^4-\E|Z|^4|
\leq \sum_{j,k=1}^d\Delta_{j,k}+d\sum_{j=1}^dC_jM^{p_j}\mcl{M}(f_j).
} 

For $\nu>0$, we denote by $\gamma_\pm(\nu)$ the law of the random variable $\pm(\xi-\nu)/\sqrt{\nu}$, where $\xi$ is a gamma variable with shape $\nu$ and rate 1. 
Then, setting $s_i:=\E X_i^3$ for each $i\in[n]$, we construct independent random variables $Y=\{Y_1,\dots,Y_n\}$ so that they are independent of $X$ and satisfy 
\[
Y_i\sim\left\{
\begin{array}{cl}
N(0,1)  & \text{if }s_i=0, \\
\gamma_+(4/s_i^2)  &  \text{if }s_i>0,   \\
\gamma_-(4/s_i^2)   &  \text{if }s_i<0.
\end{array}
\right.
\]
By construction we have $\E X_i^r=\E Y_i^r$ for any $i\in[n]$ and $r=1,2,3$. 
By the Schwarz inequality, we also have
\ban{
s_i^2&\leq\E X_i^2\E X_i^4=\E X_i^4,\label{est-s2}\\
\E Y_i^4&=3(1+s_i^2/2)\leq3(1+\E X_i^4/2)\leq\frac{9}{2}\E X_i^4.\label{est-y4}
}
We define the random vector $\tilde W$ in $\mathbb{R}^d$ by
\[
\tilde W_j:=\sum_{i_1,\dots,i_{p_j}=1}^nf_j(i_1,\dots,i_{p_j})Y_{i_1}\cdots Y_{i_p},\qquad j=1,\dots,d.
\]

\begin{lemma}\label{npr4.3}
For any $j,k\in[d]$, it holds that
\[
|\E W_j^2W_k^2-\E\tilde{W}_j^2\tilde{W}_k^2|
\leq C_{j}M^{p_j}\mcl{M}(f_j)+C_{k}M^{p_k}\mcl{M}(f_k).
\]
\end{lemma}

\begin{proof}
The proof is a minor modification of \cite[Lemma 4.3]{NoPeRe10}. We begin by introducing some notation. Given a sequence of random variables $Z=\{Z_1,\dots,Z_n\}$, we set
\[
Q_j(Z)=\sum_{i_1,\dots,i_{p_j}=1}^nf_j(i_1,\dots,i_{p_j})Z_{i_1}\cdots Z_{i_p},\qquad j=1,\dots,d.
\]
For $i\in\{0,1,\dots,n\}$ and $j\in[d]$, we define
\[
Z^{(i)}=\{Z_1^{(i)},\dots,Z_n^{(i)}\}:=\{X_1,\dots,X_i,Y_{i+1},\dots,Y_n\}
\]
and
\begin{align*}
U_{i,j}&:=\sum_{\begin{subarray}{c}
i_1,\dots,i_{p_j}=1\\
i_1\neq i,\dots,i_{p_j}\neq i
\end{subarray}}^{n}
f_j(i_1,\dots,i_{p_j})Z_{i_1}^{(i)}\cdots Z_{i_{p_j}}^{(i)},&
V_{i,j}&:=\sum_{\begin{subarray}{c}
i_1,\dots,i_{p_j}=1\\
\exists l:i_l= i
\end{subarray}}^{n}
f_j(i_1,\dots,i_{p_j})\prod_{l:i_l\neq i}Z_{i_l}^{(i)}.
\end{align*}
By construction $U_{i,j}$ and $V_{i,j}$ are independent of $X_i$ and $Y_i$. Also, we have $Q_j(Z^{(i)})=U_{i,j}+X_iV_{i,j}$ and $Q_j(Z^{(i-1)})=U_{i,j}+Y_iV_{i,j}$. Therefore, we have for any $j,k\in[d]$
\begin{align*}
&\E Q_j(Z^{(i)})^2Q_k(Z^{(i)})^2\\
&=\E(U_{i,j}+X_iV_{i,j})^2(U_{i,k}+X_iV_{i,k})^2\\
&=\E(U_{i,j}^2+X_i^2V_{i,j}^2+2U_{i,j}X_iV_{i,j})(U_{i,k}^2+X_i^2V_{i,k}^2+2U_{i,k}X_iV_{i,k})\\
&=\E[U_{i,j}^2U_{i,k}^2+U_{i,j}^2V_{i,k}^2+V_{i,j}^2U_{i,k}^2
+X_i^4V_{i,j}^2V_{i,k}^2\\
&\qquad+2X_i^3V_{i,j}^2U_{i,k}V_{i,k}
+2X_i^3V_{i,k}^2U_{i,j}V_{i,j}
+4U_{i,j}V_{i,j}U_{i,k}V_{i,k}]
\end{align*}
and a similar expression for $\E Q_j(Z^{(i-1)})^2Q_k(Z^{(i-1)})^2$ with replacing $X_i$ by $Y_i$. Hence, noting $\E X_i^3=\E Y_i^3$, we obtain 
\ba{
&|\E Q_j(Z^{(i)})^2Q_k(Z^{(i)})^2-\E Q_j(Z^{(i-1)})^2Q_k(Z^{(i-1)})^2|\\
&=|\E(X_i^4-Y_i^4)\E V_{i,j}^2V_{i,k}^2|
\leq \max\{\E X_i^4,\E Y_i^4\}\frac{\E V_{i,j}^4+\E V_{i,k}^4}{2}.
}
We have by Lemma 4.2 of \cite{NoPeRe10}
\ba{
\E V_{i,j}^4&\leq\max_{1\leq i\leq n}\max\{\E X_i^4,\E Y_i^4\}^{p_j-1}(2\sqrt3)^{4(p_j-1)}(\E V_{i,j}^2)^2\\
&\leq C_j\max_{1\leq i\leq n}\max\{\E X_i^4,\E Y_i^4\}^{p_j-1}\influence_i(f_j)^2.
}
Combining these estimates with \eqref{est-y4}, we obtain
\ba{
&|\E Q_j(Z^{(i)})^2Q_k(Z^{(i)})^2-\E Q_j(Z^{(i-1)})^2Q_k(Z^{(i-1)})^2|\\
&\leq C_{j}M^{p_j}\influence_i(f_j)^2+C_{k}M^{p_k}\influence_i(f_k)^2.
}
Since $W_j=Q_j(X)=Q_j(Z^{(n)})$, $\tilde W_j=Q_j(Y)=Q_j(Z^{(0)})$ and
\[
\sum_{i=1}^n\influence_i(f_j)=\sum_{i_1,\dots,i_{p_j}=1}^nf_j(i_1,\dots,i_{p_j})^2
=\frac{1}{p_j!},
\]
we conclude
\ba{
|\E W_j^2W_k^2-\E\tilde{W}_j^2\tilde{W}_k^2|
&\leq\sum_{i=1}^n|\E Q_j(Z^{(i)})^2Q_k(Z^{(i)})^2-\E Q_j(Z^{(i-1)})^2Q_k(Z^{(i-1)})^2|\\
&\leq C_{j}M^{p_j}\mcl{M}(f_j)+C_{k}M^{p_k}\mcl{M}(f_k).
}
This completes the proof.
\end{proof}

We turn to the proof of \eqref{homo:aim}. 
We have by Lemma \ref{npr4.3}
\ben{\label{homo-1}
|\E|W|^4-\E|\tilde{W}|^4|
\leq\sum_{j,k=1}^d|\E W_j^2W_k^2-\E\tilde{W}_j^2\tilde{W}_k^2|
\leq d\sum_{j=1}^dC_jM^{p_j}\mcl{M}(f_j).
}
Using Eq.(4.2) of \cite{NoRo14}, we obtain
\ben{\label{homo0}
\E|\tilde{W}|^4-\E|Z|^4
=\sum_{j,k=1}^d(\E\tilde{W}_j^2\tilde{W}_k^2-\Sigma_{jj}\Sigma_{kk}-2\Sigma_{jk}^2)
=\sum_{j,k=1}^d(\Cov(\tilde{W}_j^2,\tilde{W}_k^2)-2\Sigma_{jk}^2).
}
We have by Eq.(5.17) of \cite{Ko19} 
\bmn{\label{homo1}
|\Cov(\tilde{W}_j^2,\tilde{W}_k^2)-2\Sigma_{jk}^2|
\leq 1_{\{p_j\ne p_k\}}\sqrt{\E\tilde{W}_{j\wedge k}^4}\sqrt{\E\tilde{W}_{j\vee k}^4-3}\\
+1_{\{p_j=p_k\}}\left\{2\sqrt{\E\tilde{W}_j^4-3}\sqrt{\E\tilde{W}_k^4-3}
+(2^{-p_j}\ol{v}_n^{p_j}-1)\sqrt{\mcl{M}(f_j)\mcl{M}(f_k)}\right\},
}
where $\ol{v}_n:=\max_{1\leq i\leq n}(2+s_i^2/2)$. From \eqref{est-s2} we obtain
\ben{\label{homo2}
2^{-p_j}\ol{v}_n^{p_j}\leq C_j\left(1+\max_{1\leq i\leq n}s_i^{2p_j}\right)\leq C_jM^{p_j}.
}
Also, we have by Lemma 4.2 of \cite{NoPeRe10} and \eqref{est-y4}
\ben{\label{homo3}
\E\tilde{W}_j^4\leq C_jM^{p_j}(\E\tilde{W}_j^2)^2=C_jM^{p_j}.
}
Moreover, Lemma \ref{npr4.3} yields
\ben{\label{homo4}
\E\tilde{W}_j^4\leq\E W_j^4+C_jM^{p_j}\mcl{M}(f_j).
}
Inserting \eqref{homo1}--\eqref{homo4} into \eqref{homo0} and using the AM-GM inequality, we obtain
\bm{
|\E|\tilde{W}|^4-\E|Z|^4|
\leq \sum_{j,k=1}^d1_{\{p_j<p_k\}}C_jM^{p_j/2}\sqrt{(\E W_k^4-3)+C_kM^{p_k}\mcl{M}(f_k)}\\
+2\sum_{j,k=1}^d1_{\{p_j=p_k\}}\left((\E W_j^4-3)+C_jM^{p_j}\mcl{M}(f_j)\right).
} 
Combining this estimate with \eqref{homo-1}, we obtain \eqref{homo:aim}.
\end{proof}

Finally, we apply the continuous version of the Wasserstein bound in Theorem~\ref{t4} to obtain Theorem~\ref{t5}.

\begin{proof}[Proof of Theorem~\ref{t5}]
From the proof of Theorem 1.7 in \cite{DoViZh18}, there is a family of $d$-dimensional random vectors $(W_t)_{t>0}$ satisfying the assumptions of Theorem \ref{t4} with $\Lambda=\diag(q_1,\dots,q_d)$, $R=0$, $\rho_j(W)=2(4q_j-3)\left(\E W_j^4-3(\E W_j^2)^2\right)$ and some $d\times d$ random matrix $S$. Moreover, from Eqs.(4.2)--(4.3) in \cite{DoViZh18}, this matrix $S$ satisfies
\[
\E\|S\|_{H.S.}\leq(2q_d-1)\sqrt{\E[|W|^4-|Z|^4]}.
\]
Consequently, we obtain
\begin{multline*}
d_{\mathcal{W}}(W,Z)
\leq\norm{\Sigma^{-1/2}}_{op}\frac{2q_d-1}{q_1\sqrt{2\pi}}\sqrt{\E[|W|^4-|Z|^4]}\\
+\norm{\Sigma^{-1/2}}_{op}^{3/2}\left(\frac{\pi}{8}\right)^{1/4}\tr(\Sigma)^{1/4}\sqrt{d}\sqrt{\frac{2}{q_1}\sum_{j=1}^d(4q_j-3)\left(\E W_j^4-3(\E W_j^2)^2\right)}.
\end{multline*}
Since $\sqrt{2\pi}\geq2$ and $\max_jq_j=q_d$, we obtain \eqref{pois1}. 
Finally, if $q_1=\cdots=q_d$, Lemma 4.1 in \cite{DoViZh18} yields
\[
\sqrt{\E[|W|^4-|Z|^4]}\leq\sqrt2\sum_{j=1}^d\sqrt{\E W_j^4-3(\E W_j^2)^2}
\leq\sqrt{2d\sum_{j=1}^d\left(\E W_j^4-3(\E W_j^2)^2\right)}.
\]
Hence \eqref{pois2} holds true. 
\end{proof}

\appendix

\section{Appendix: Optimal Wasserstein Bound for Wishart Matrices}

For multivariate normal approximation of Wishart matrices, a direct application of Theorem \ref{t3}
yields a sub-optimal Wasserstein bound (cf. Remark \ref{l2remark}). 
However, we can modify the proof of Theorem \ref{t3} to further exploit the independence structure in Wishart matrices and obtain the optimal Wasserstein bound assuming finite sixth moments. More precisely, we obtain:
\begin{theorem}\label{tA1}
Under the conditions of Theorem \ref{t2}, if we assume in addition that $\E X_{11}^6<\infty$, then
\be{
d_{\mathcal{W}}(W,Z)\leq C\sqrt{\frac{n^3}{d}}\vee\left(\frac{n^3}{d}\right)^{2/3},
}
where $C$ is a constant only depending on $\E X_{11}^6$.
\end{theorem}

Below, we first give the modified version of Theorem \ref{t3} (cf. Theorem \ref{tA2}) using truncation,
we then use it to prove Theorem \ref{tA1}.
In the remainder of this appendix, we use $C$ to denote positive constants, which may depend on $\E X_{11}^6$ and may be different in different expressions.

\subsection{A modified Wasserstein bound}

We use the idea of truncation from \cite{Bo20} to modify the proof of Theorem \ref{t3} as follows.
We assume that 
\be{
\Sigma=I_d,\quad \Lambda=\lambda I_d, \quad R=0
}
as in the case for Wishart matrices.
Recall that in \eq{14}, we need to bound
\be{
\int_\eps^{\pi/2}|\E[\mathscr{S}\tilde h_\alpha(W)]|\tan\alpha ~d\alpha,
}
where for a 1-Lipschitz function $h$,
\be{
\mathscr{S} \tilde h_\alpha(w)=\Delta \tilde h_\alpha (w)- w\cdot \nabla \tilde h_\alpha(w)
}
and
\be{
\tilde h_\alpha(w)=\int_{\mathbb{R}^d} h(w \cos \alpha+z \sin\alpha) \phi_d(z) dz.
}
We begin by applying a truncation to \eq{start} with $f=\tilde h_\alpha$ as follows.
Note that $\tilde h_\alpha$ is infinitely differentiable for $\alpha\in (0,\pi/2)$.
We have
\begin{align*}
0&=\frac{1}{2}\E[\lambda^{-1}D\cdot(\nabla \tilde h_\alpha (W')+\nabla \tilde h_\alpha(W))1_{\{|D|\leq \tan \alpha\}}]\\
&=\E\left[\frac{1}{2}\lambda^{-1}D\cdot(\nabla \tha(W')-\nabla \tha(W))1_{\{|D|\leq \tan \alpha\}}+\lambda^{-1}D\cdot\nabla \tha(W)1_{\{|D|\leq \tan \alpha\}}\right]\\
&=\E\left[\langle \frac{D D^\top}{2\lambda} , \text{Hess} \tha (W)\rangle_{H.S.}+\Xi_\alpha
- W\cdot \nabla \tha (W) -  \lambda^{-1} D\cdot \nabla \tha (W) 1_{\{|D|>\tan \alpha\}}    \right]\\
&=\E \Big[ \Delta \tha(W)+\langle E, \text{Hess} \tha (W)\rangle_{H.S.}-  \langle \frac{D D^\top}{2\lambda} , \text{Hess} \tha (W)\rangle_{H.S.}1_{\{|D|>\tan \alpha\}} \\
&\qquad \quad  +\Xi_\alpha
- W\cdot \nabla \tha (W) -  \lambda^{-1} D\cdot \nabla \tha (W) 1_{\{|D|>\tan \alpha\}}    \Big],
\end{align*}
where 
\be{
\Xi_\alpha=\frac{1}{2\lambda}\sum_{j,k,l=1}^d D_jD_kD_lU\partial_{jkl}\tha(W+(1-U)D)1_{\{|D|\leq \tan \alpha\}}.
}
This implies
\besn{\label{fA2}
|\E \mathscr{S}\tha(W)|\leq& |\E \lambda^{-1}  D\cdot \nabla \tha (W) 1_{\{|D|>\tan \alpha\}}| + |\E\langle E, \text{Hess} \tha (W)\rangle_{H.S.}|\\
&+ |\E \langle \frac{D D^\top}{2\lambda} , \text{Hess} \tha (W)\rangle_{H.S.}1_{\{|D|>\tan \alpha\}}|+|\E \Xi_\alpha|\\
=: & R_{1\alpha}+R_{2\alpha}+R_{3\alpha} +R_{4\alpha}.
}
We bound the four terms one by one as follows.

Because $|\nabla h(w)|\leq 1$, we have
\be{
R_{1\alpha}\leq \lambda^{-1} \E|\E[D1_{\{|D|>\tan \alpha\}}|\mathcal{G}]|.
}
From \eq{deriv2}, we have
\be{
R_{2\alpha}\leq \frac{C}{\tan \alpha} \E\norm{E}_{H.S.}.
}
For $R_{3\alpha}$, again using \eq{deriv2}, we have
\be{
R_{3\alpha}\leq \frac{C}{\lambda \tan \alpha} \E \norm{\E[D D^\top 1_{\{|D|>\tan \alpha\}}\mid\mcl{G}]}_{H.S.}.
}
We will use the following lemma which will be proved at the end of this subsection.
\begin{lemma}\label{lA1}
Let $Y=(Y_{ij})_{1\leq i, j\leq d}$ be a $d\times d$ positive semidefinite symmetric random matrix. Let $F$ and $G$ be two random variables such that $|F|\leq G$. Suppose that $\E|Y_{ij} F|<\infty$ for all $i,j=1,\dots, d$. 
Let $\mathcal{G}$ be an arbitrary $\sigma$-field.
Then we have
\be{
\norm{\E[YF|\mathcal{G}]}_{H.S.}\leq \norm{\E[YG|\mathcal{G}]}_{H.S.}.
}
\end{lemma}
From Lemma \ref{lA1}, 
\be{
R_{3\alpha}\leq \frac{C}{\lambda \tan^{3}\alpha} \E \norm{\E[D D^\top |D|^2\mid\mcl{G}]}_{H.S.	}.
}
For $R_{4\alpha}$, following the symmetry argument as in \eq{eq:xi}, we have
\be{
R_{4\alpha}=\left| \frac{1}{4\lambda} \E\langle U D_\alpha^{\otimes3}, \nabla^3 \tha(W+(1-U)D_\alpha)-\nabla^3 \tha (W+UD_\alpha)    \rangle    \right|,
}
where $D_\alpha:=D1_{\{|D|\leq \tan \alpha\}}$. 
By Taylor's expansion (cf. Lemma 1 of \cite{Bo20} and its proof),
\bes{
&\E\langle U D_\alpha^{\otimes3}, \nabla^3 \tha(W+(1-U)D_\alpha)-\nabla^3 \tha (W+UD_\alpha)    \rangle\\
=& \sum_{k=4}^\infty \frac{1}{(k-3)!} \E\langle U((1-U)^{k-3}-U^{k-3}) D_\alpha^{\otimes k}, \nabla^k \tha(W) \rangle.
}
This implies, by the Cauchy-Schwarz inequality,
\bes{
R_{4\alpha}\leq & \frac{\cos \alpha}{4\lambda}\E \sqrt{\sum_{k=4}^\infty \frac{(\tan \alpha)^{2k-2}}{\cos^2 \alpha (k-1)!} |\nabla^k \tha(W)|^2}\\
&\times \sqrt{\sum_{k=4}^\infty \frac{(k-1)!}{((k-3)!)^2 (\tan \alpha)^{2k-2}}|\E[D_\alpha^{\otimes k}|\mathcal{G}]|^2 },
}
where $|\cdot|$ denotes the Euclidean norm by regarding the tensor as a $d^k$-vector.
For every $w\in \mathbb{R}^d$, we have by a similar argument to the proof of Eq.(17) in \cite{Bo20},
\be{
\sum_{k=4}^\infty \frac{(\tan \alpha)^{2k-2}}{\cos^2 \alpha (k-1)!} |\nabla^k \tha(W)|^2\leq M_1(h)\leq 1.
}
We will use the following lemma which will be proved at the end of this subsection.
\begin{lemma}\label{lA2}
Let $Y$ be a random vector in $\mathbb{R}^d$ such that $\E|Y|^k<\infty$ for some integer $k\geq 2$. 
Let $\mathcal{G}$ be an arbitrary $\sigma$-field.
Then
\be{
|\E [Y^{\otimes k}|\mathcal{G}]|\leq \norm{\E[Y Y^\top|Y|^{k-2}|\mathcal{G}]}_{H.S.}.
}
\end{lemma}
From Lemma \ref{lA2} and Lemma \ref{lA1} and using $|D_\alpha|\leq \tan \alpha$, we obtain for $k\geq 4$
\bes{
&|\E[D_\alpha^{\otimes k}|\mathcal{G}]|^2\leq |\E[D_\alpha D_\alpha^\top \norm{D_\alpha|^{k-2}|\mathcal{G}]}_{H.S.}^2\\
\leq & (\tan \alpha)^{2k-8} \norm{\E[D_\alpha D_\alpha^\top |D_\alpha|^2|\mathcal{G}]}_{H.S.}^2
\leq (\tan \alpha)^{2k-8} \norm{\E[D D^\top |D|^2|\mathcal{G}]}_{H.S.}^2.
}
We infer that
\bes{
&\sum_{k=4}^\infty \frac{(k-1)!}{((k-3)!)^2 (\tan \alpha)^{2k-2}}|\E[D_\alpha^{\otimes k}|\mathcal{G}]|^2\\
\leq & \left(\sum_{k=4}^\infty \frac{(k-1)!}{((k-3)!)^2} \right) \frac{1}{\tan^6 \alpha} \norm{\E[D D^\top |D|^2|\mathcal{G}]}_{H.S.}^2\\
\leq & \frac{C}{\tan^6 \alpha}\norm{\E[D D^\top |D|^2|\mathcal{G}]}_{H.S.}^2,
}
and hence
\be{
R_{4\alpha}\leq \frac{C}{\lambda \tan^3\alpha} \E\norm{\E[D D^\top |D|^2|\mathcal{G}]}_{H.S.}.
}
Combining the bounds for $R_{1\alpha}$--$R_{4\alpha}$, we have, from \eq{14} and \eq{fA2},
\bes{
d_{\mathcal{W}}(W, Z)\leq&\int_\eps^{\pi/2}|\E[\mathscr{S}\tilde h_\alpha(W)]|\tan\alpha ~d\alpha
+2\sqrt{d}\sin\frac{\eps}{2}\\
\leq & \lambda^{-1}\int_\eps^{\pi/2} \E|\E[D1_{|D|>\tan \alpha}|\mathcal{G}]|\tan\alpha ~d\alpha+C\E\norm{E}_{H.S.}\\
&+C\lambda^{-1}\int_\eps^{\pi/2}\frac{1}{ \tan^2\alpha} \E\norm{\E[D D^\top |D|^2|\mathcal{G}]}_{H.S.} d\alpha
+2\sqrt{d}\sin\frac{\eps}{2}.
}
Optimizing $\eps$, we obtain:
\begin{theorem}\label{tA2}
Under the conditions of Theorem \ref{t3}, if we assume in addition that 
\be{
\Sigma=I_d,\quad \Lambda=\lambda I_d, \quad R=0,
}
then
\bes{
d_{\mathcal{W}}(W, Z)\leq& \lambda^{-1}\int_0^{\pi/2} \E|\E[D1_{|D|>\tan \alpha}|\mathcal{G}]|\tan\alpha ~d\alpha
+C\E\norm{E}_{H.S.}\\
&+C d^{1/4} \sqrt{\lambda^{-1}\E\norm{\E[D D^\top |D|^2|\mathcal{G}]}_{H.S.} }.
}
\end{theorem}

\begin{remark}
The assumptions $\Sigma=I_d$ and $R=0$ are not essential and can be easily relaxed.
The assumption $\Lambda=\lambda I_d$ is crucial to be able to apply Lemmas \ref{lA1} and \ref{lA2}.
Although Theorem \ref{tA2} is in a slightly more complicated form than Theorem \ref{t3}, it is in terms of conditional expectations $\E[\cdot |\mathcal{G}]$, which enables us to exploit further the independence structure in Wishart matrices. We note, however, Theorem \ref{tA2} is not strictly better than Theorem \ref{t3} even under these additional assumptions and does not improve the rate (e.g., for sums of independent random vectors) in general.
\end{remark}

\begin{proof}[Proof of Lemma \ref{lA1}]
Given a $d\times d$ symmetric matrix $A$, we write the ordered eigenvalues of $A$ as $\lambda_1(A)\leq\cdots\leq\lambda_d(A)$. 
For any $u\in\mathbb{R}^d$, we have
\[
u^\top(\E[YG\mid\mcl{G}]-\E[YF\mid\mcl{G}])u=\E[u^\top Y u\cdot(G-F)\mid\mcl{G}]\geq0.
\]
Therefore, $\E[YG\mid\mcl{G}]-\E[YF\mid\mcl{G}]$ is positive semidefinite. Similarly, we can prove $\E[YG\mid\mcl{G}]-(-\E[YF\mid\mcl{G}])$ is positive semidefinite. Thus, we have $|\lambda_j(\E[YF\mid\mcl{G}])|\leq\lambda_j(\E[YG\mid\mcl{G}])$ for all $j$ by Corollary 7.7.4 of \cite{HoJo13}. Hence we conclude
\[
\|\E[YF\mid\mcl{G}]\|_{H.S.}^2=\sum_{j=1}^d\lambda_j(\E[YF\mid\mcl{G}])^2\leq\sum_{j=1}^d\lambda_j(\E[YG\mid\mcl{G}])^2=\|\E[YG\mid\mcl{G}]\|_{H.S.}^2.
\]
This completes the proof.
\end{proof}

\begin{proof}[Proof of Lemma \ref{lA2}]
Taking the regular conditional probability distribution of $Y$ given $\mcl{G}$ instead of the original probability measure, it suffices to consider the unconditional case. 
Let $Y'$ be an independent copy of $Y$. Then
\ba{
\E(Y\cdot Y')^k&=\sum_{j_1,\dots,j_k=1}^d\E[Y_{j_1}Y'_{j_1}\cdots Y_{j_k}Y'_{j_k}]
=\sum_{j_1,\dots,j_k=1}^d\E[Y_{j_1}\cdots Y_{j_k}]\E[Y'_{j_1}\cdots Y'_{j_k}]\\
&=\sum_{j_1,\dots,j_k=1}^d\E[Y_{j_1}\cdots Y_{j_k}]^2
=|\E[Y^{\otimes k}]|^2.
}
Hence, we have by the Cauchy-Schwarz inequality
\ba{
|\E Y^{\otimes k}|^2
&\leq\E(Y\cdot Y')^2|Y|^{k-2}|Y'|^{k-2}
=\sum_{i,j=1}^d\E[Y_iY'_iY_jY'_j|Y|^{k-2}|Y'|^{k-2}]\\
&=\sum_{i,j=1}^d\E[Y_iY_j|Y|^{k-2}]\E[Y'_iY'_j|Y'|^{k-2}]
=\sum_{i,j=1}^d\E[Y_iY_j|Y|^{k-2}]^2\\
&=\|\E[YY^\top|Y|^{k-2}]\|_{H.S.}^2.
}
This completes the proof.
\end{proof}

\subsection{Proof of Theorem \ref{tA1}}
Now we apply Theorem \ref{tA2} to multivariate normal approximation of Wishart matrices.
Recall the basic notation from Sections \ref{sec2.1} and \ref{sec3.2}.
Note that now $d$ denotes the number of columns of the i.i.d.\ matrix $X=\{X_{ik}: 1\leq i\leq n, 1\leq k\leq d\}$ and the dimension is ${n \choose 2}$.

In \eq{12}, we have proved that
\be{
\E\norm{E}_{H.S.}\leq C\sqrt{\frac{n^3}{d}}.
}
We now first bound
\be{
\lambda^{-1}\int_0^{\pi/2} \E|\E[D1_{|D|>\tan \alpha}|\mathcal{G}]|\tan\alpha ~d\alpha.
}
Denote $D^{ik}$ to be the difference $W'-W$ given the random indices in the construction of exchangeable paris at the beginning of proof of Theorem \ref{t2} as $I=i, K=k$ for some $1\leq i\leq n$ and $1\leq k\leq d$.
Then $D^{ik}$ is regarded as an ${n \choose 2}$-vector $(D^{ik}_{lm}: 1\leq l<m\leq n)$
and
\ben{\label{fA5}
D^{ik}_{lm}=\frac{1}{\sqrt{d}}\left[\delta_{il} (X_{ik}^*-X_{ik}) X_{mk} +  \delta_{im} (X_{ik}^*-X_{ik}) X_{lk} \right].
}
Note that
\be{
\E[D1_{\{|D|>\tan \alpha\}}|\mathcal{G}]=\frac{1}{nd}\sum_{i=1}^n \sum_{k=1}^d \E[D^{ik}1_{|D^{ik}|>\tan \alpha}|\mathcal{G}]
}
and each conditional expectation has mean 0 by exchangeability.
Also, from its expression above, $D^{ik}_{lm}$ are independent across different $k$'s.
Therefore,
\besn{\label{fA3}
&\lambda^{-1}\E|\E[D1_{|D|>\tan \alpha}|\mathcal{G}]|\\
\leq & \frac{1}{2}\sqrt{\sum_{1\leq l<m\leq n}\sum_{k=1}^d \Var \left(\sum_{i=1}^n D^{ik}_{lm}1_{\{|D^{ik}>\tan \alpha|\}} \right)}\\
= & \frac{1}{2}\sqrt{\sum_{1\leq l<m\leq n}\sum_{k=1}^d \Var \left(D^{lk}_{lm}1_{\{|D^{lk}>\tan \alpha|\}}+
D^{mk}_{lm}1_{\{|D^{mk}>\tan \alpha|\}} \right)},
}
where the last equality is from the fact that $D^{ik}_{lm}$ is non-zero only if $i=l$ or $m$.
Now,
\bes{
&\frac{1}{2}\sqrt{\sum_{1\leq l<m\leq n}\sum_{k=1}^d \Var \left(D^{lk}_{lm}1_{\{|D^{lk}>\tan \alpha|\}} \right)}\\
\leq &\frac{1}{2\tan \alpha}\sqrt{\sum_{1\leq l<m\leq n}\sum_{k=1}^d \E \left[ (D^{lk}_{lm})^2 |D^{lk}|^2 \right]}\\
= &\frac{1}{2\tan \alpha}\sqrt{\sum_{1\leq l<m\leq n}\sum_{k=1}^d \E \left[ (D^{lk}_{lm})^2 
\left(\sum_{u: u<l} (D^{lk}_{ul})^2 +\sum_{v: v>l}(D^{lk}_{lv})^2    \right) \right]}\\
\leq &\frac{C}{\tan \alpha} \sqrt{n^2 d \frac{1}{d^2} n}=\frac{C}{\tan \alpha} \sqrt{\frac{n^3}{d}}.
}
Together with the same bound for the variance of the second term in \eq{fA3}, we obtain
\be{
\lambda^{-1}\int_0^{\pi/2} \E|\E[D1_{|D|>\tan \alpha}|\mathcal{G}]|\tan\alpha ~d\alpha\leq C\sqrt{\frac{n^3}{d}}.
}

Next, we bound
\be{
{n\choose 2}^{1/4} \sqrt{\lambda^{-1}\E\norm{\E[D D^\top |D|^2|\mathcal{G}]}_{H.S.} },
}
which is further bounded by 
\ben{\label{fA4}
{n\choose 2}^{1/4}  \left(\norm{\lambda^{-1} \E[DD^\top |D|^2]}_{H.S.}^{1/2}  
+\sqrt{\lambda^{-1}\E\norm{\E[D D^\top |D|^2|\mathcal{G}]-\E[DD^\top |D|^2]}_{H.S.}} \right).
}
For the first term in \eq{fA4}, we have
\bes{
&\norm{\lambda^{-1} \E[DD^\top |D|^2]}_{H.S.}^2
=\sum_{1\leq l<m\leq n\atop 1\leq u<v\leq n}(\lambda^{-1}\sum_{1\leq p<q\leq n} \E[D_{lm}D_{uv}D_{pq}^2])^2\\
=&\frac{1}{4}\sum_{1\leq l<m\leq n\atop 1\leq u<v\leq n} \left( \sum_{i=1}^n \sum_{k=1}^d \sum_{1\leq p<q\leq n}\E [D^{ik}_{lm}D^{ik}_{uv}(D^{ik}_{pq})^2]  \right)^2.
}
From \eq{fA5}, for the expectation to be non-zero, we must have $(l,m)=(u,v)$, $i=l$ or $m$, and $p$ or $q=i$. 
Hence,
\be{
{n\choose 2}^{1/4}  \left(\norm{\lambda^{-1} \E[DD^\top |D|^2]}_{H.S.}^2  \right)^{1/4}
\leq C {n\choose 2}^{1/4} \left( n^2(d n \frac{1}{d^2})^2  \right)^{1/4}\leq C\sqrt{\frac{n^3}{d}} .
}
For the second term in \eq{fA4},
note that
\ben{\label{fA4-2nd}
\lambda^{-1}\E\norm{\E[D D^\top |D|^2|\mathcal{G}]-\E[DD^\top |D|^2]}_{H.S.}
=\frac{1}{2}\E\left\|\sum_{k=1}^d\sum_{i=1}^nY^{ik}\right\|_{H.S.},
}
where
\[
Y^{ik}:=\E[D^{ik} (D^{ik})^\top |D^{ik}|^2|\mathcal{G}]-\E[D^{ik} (D^{ik})^\top |D^{ik}|^2].
\]
Since $Y^{ik}$ are independent across different $k$'s, we have by the symmetrization inequality (cf.~Lemma 6.3 in \cite{LeTa91})
\ba{
\E\left\|\sum_{k=1}^d\sum_{i=1}^nY^{ik}\right\|_{H.S.}
\leq2\E\left\|\sum_{k=1}^d\eps_k\sum_{i=1}^nY^{ik}\right\|_{H.S.},
}
where $\epsilon=\{\epsilon_1,\dots,\epsilon_d\}$ is a sequence of independent Rademacher variables independent of everything else. 
Then, we obtain by Jensen's inequality
\ba{
\E\left\|\sum_{k=1}^d\sum_{i=1}^nY^{ik}\right\|_{H.S.}
&\leq2\E\sqrt{\E\left[\left\|\sum_{k=1}^d\eps_k\sum_{i=1}^nY^{ik}\right\|_{H.S.}^2\mid X\right]}
=2\E\sqrt{\sum_{k=1}^d\left\|\sum_{i=1}^nY^{ik}\right\|_{H.S.}^2}\\
&\leq2\left(\E\left[\left(\sum_{k=1}^d\left\|\sum_{i=1}^nY^{ik}\right\|_{H.S.}^2\right)^{3/4}\right]\right)^{2/3}\\
&=2\left(\E\left[\left(\sum_{k=1}^d\sum_{1\leq l<m\leq n\atop 1\leq u<v\leq n}\left(\sum_{i=1}^nY^{ik}_{lm,uv}\right)^2\right)^{3/4}\right]\right)^{2/3},
}
where
\[
Y^{ik}_{lm,uv}:=\E[D^{ik}_{lm}D^{ik}_{uv} |D^{ik}|^2|\mathcal{G}]-\E[D^{ik}_{lm} D^{ik}_{uv} |D^{ik}|^2].
\]
Using the elementary inequality $(\sum_jx_j)^{3/4}\leq \sum_jx_j^{3/4}$ for $x_j\geq0$, we conclude
\ba{
\E\left\|\sum_{k=1}^d\sum_{i=1}^nY^{ik}\right\|_{H.S.}
\leq2\left(\sum_{k=1}^d\sum_{1\leq l<m\leq n\atop 1\leq u<v\leq n}\E\left[\left|\sum_{i=1}^nY^{ik}_{lm,uv}\right|^{3/2}\right]\right)^{2/3}.
}
Now, note that from the expression \eq{fA5}, for $D^{ik}_{lm} D^{ik}_{uv}$ to be non-zero, $\{l,m\}$ and $\{u, v\}$ must have at least one element in common and $i$ must be such a common element. There are $\sim n^3$ such combinations.
For a typical term with some $k=1,\dots, d$,
\ba{
\E\left[\left|Y^{lk}_{lm,lv}\right|^{3/2}\right]
&\leq C\E\left[\left|D^{lk}_{lm}D^{lk}_{lv} |D^{lk}|^2\right|^{3/2}\right]\\
&=C\E\left[\left|\sum_{q: q>l} D^{lk}_{lm} D^{lk}_{lv} (D^{lk}_{lq})^2   +\sum_{p: p<l} D^{lk}_{lm} D^{lk}_{lv} (D^{lk}_{pl})^2\right|^{3/2}\right]\\
&=C\E\left[\left|\frac{1}{d^2} \sum_{q: q\ne l} (X_{lk}^*-X_{lk})^4 X_{mk} X_{vk} X_{qk}^2\right|^{3/2}\right]
}
Since $l\neq m$ and $l\neq v$, we obtain 
\ba{
\E\left[\left|Y^{lk}_{lm,lv}\right|^{3/2}\right]
&\leq\frac{C}{d^3}\E\left[(X_{lk}^*-X_{lk})^6\right]\E\left[\left|X_{mk} X_{vk} \sum_{q: q\ne l}   X_{qk}^2\right|^{3/2}\right]\\
&\leq\frac{C}{d^3}\E\left[X_{lk}^6\right]^{3/2}\sqrt{\E\left[\left(\sum_{q: q\ne l}   X_{qk}^2\right)^{3}\right]}
\leq\frac{C}{d^3}n^{3/2},
}
assuming the existence of finite sixth moment. Hence we conclude
\ba{
{n\choose 2}^{1/4} \sqrt{\E\left\|\sum_{k=1}^d\sum_{i=1}^nY^{ik}\right\|_{H.S.}}
\leq Cn^{1/2}\left(dn^3\frac{n^{3/2}}{d^3}\right)^{1/3}
=C\frac{n^2}{d^{2/3}}=C\left(\frac{n^3}{d}\right)^{2/3}.
}
Therefore, the second term in \eq{fA4} is bounded by $C(n^3/d)^{2/3}$ due to \eqref{fA4-2nd}. 
Theorem \ref{tA1} is proved by combining the above bounds.

\section*{Acknowledgements}

We thank the two anonymous referees for their careful reading of the manuscript and for their valuable suggestions which led to many improvements.
Fang X. was partially supported by Hong Kong RGC ECS 24301617 and GRF 14302418 and 14304917, a CUHK direct grant and a CUHK start-up grant. 
Koike Y. was partially supported by JST CREST Grant Number JPMJCR14D7 and JSPS KAKENHI Grant Numbers JP17H01100, JP18H00836, JP19K13668.


\end{document}